\documentclass{my_elsarticle}

\usepackage{hyperref}
\usepackage{amsmath}
\usepackage{amssymb}
\usepackage{amsthm}
\usepackage{graphicx}
\usepackage{newlfont}
\usepackage{setspace}
\usepackage{subfig}
\usepackage{a4wide}

\newtheorem{theorem}{Theorem}[section]

\newtheorem{definition}{Definition}

\journal{To appear in \emph{Mathematical Biosciences}.
Submitted 8-Nov-2012; revised 21-May-2013; accepted 16-Oct-2013.}

\begin{document}

\begin{frontmatter}

\title{Vaccination models and optimal control strategies to dengue}

\author[label1,label2,label3]{Helena Sofia Rodrigues}
\ead{sofiarodrigues@esce.ipvc.pt}
\author[label2,label4]{M. Teresa T. Monteiro}
\ead{tm@dps.uminho.pt}
\author[label1,label5]{Delfim F. M. Torres\corref{corA}}
\ead{delfim@ua.pt}

\address[label1]{CIDMA --- Center for Research and Development in Mathematics and Applications}
\address[label2]{Center Algoritmi}
\address[label3]{School of Business Studies, Viana do Castelo Polytechnic Institute, Portugal}
\address[label4]{Department of Production and Systems, University of Minho, Portugal}
\address[label5]{Department of Mathematics, University of Aveiro, Portugal}

\cortext[corA]{Corresponding author}


\begin{abstract}
As the development of a dengue vaccine is ongoing, we
simulate an hypothetical vaccine as an extra protection
to the population. In a first phase, the vaccination process
is studied as a new compartment in the model, and
different ways of distributing the vaccines investigated:
pediatric and random mass vaccines, with distinct levels
of efficacy and durability. In a second step, the vaccination is
seen as a control variable in the epidemiological process.
In both cases, epidemic and endemic scenarios are included
in order to analyze distinct outbreak realities.
\end{abstract}


\end{frontmatter}


\section{Introduction}
\label{sec:7:0}

Since 1760, when the Swiss mathematician Daniel Bernoulli published
a study on the impact of immunization with cowpox, the process of protecting
individuals from infection by immunization has become a routine, with historical
success in reducing both mortality and morbidity \cite{Scherer2002}.
The impact of vaccination may be regarded not only as an individual
protective measure, but also as a collective one. While direct
individual protection is the major focus of a mass vaccination program,
the effects on population also contribute indirectly to other individual
protection through herd immunity, providing protection for unprotected
individuals \cite{Farrington2003}.
This means that when we have a large neighborhood of vaccinated people,
a susceptible individual has a lower probability in coming into contact with the infection,
being more difficult for diseases to spread, which decreases the relief of health facilities
and can break the chain of infection.

Dengue is a vector-borne disease that transcends international borders. It is transmitted
to humans through mosquito bite, mainly the \emph{Aedes aegypti}. In this process the
female mosquito acquires the virus while feeding on the blood of an infected person.
The blood is necessary to feed their eggs. Larvae hatch when water inundates the
eggs as a result of rains or an addition of water by people. When the larva has acquired
enough energy and size, metamorphosis is done, changing the larva into pupa. The newly formed
adult emerges from the water after breaking the pupal skin. This process could lasts
between 8 to 10 days \cite{Otero2008}.
Vector control remains the only available strategy against dengue.
Despite integrated vector control with community participation,
along with active disease surveillance and insecticides,
there are only a few examples of successful dengue prevention
and control on a national scale \cite{Cattand2006}.
Besides, the levels of resistance of \emph{Aedes aegypti} to
insecticides has increased, which implies shorter intervals between
treatments, and only few insecticide products are available in the
market due to the high costs for development and registration
and low returns \cite{Keeling2008}.

Dengue vaccines have been under development since the 1940s,
but due to the limited appreciation of global disease burden
and the potential markets for dengue vaccines, industry interest
languished throughout the 20th century. However, in recent years,
the development of dengue vaccines has dramatically accelerated
with the increase in dengue infections, as well as
the prevalence of all four circulating serotypes. Faster development
of a vaccine became a serious concern \cite{Murrell2011}. Economic analysis
are conducted to guide public support for vaccine development
in both industrialized and developing countries, including a previous
cost-effectiveness study of dengue \cite{Clark2005,Shepard2004,Shepard2009}.
The authors of these works compared the cost of the disease burden
with the possibility of making a vaccination campaign;
they suggest that there is a potential economic benefit associated with promising
dengue interventions, such as dengue vaccines and vector control innovations,
when compared to the cost associated to the disease treatments.
Constructing a successful vaccine for dengue has been challenging:
the knowledge of disease pathogenesis is insufficient and in addition
the vaccine must protect simultaneously against all serotypes in order
to not increase the level of dengue haemorrhagic fever \cite{Supriatna2008}.

Currently, the features of a dengue vaccine are mostly unknown.
Therefore, in this paper we opt to present a set of simulations
with different efficacy and different ways of distributing the vaccine.
We have also explored the vaccination process
under two different perspectives. In Section~\ref{sec:7:2} a new compartment
in the model is used and several kinds of vaccines are considered.
In Section~\ref{sec:7:3}, a  second perspective is studied using
the vaccination process as a disease control in the mathematical formulation.
In that case the theory of optimal control is applied.
Both methods assume a continuous vaccination strategy.


\section{Vaccine as a new compartment in the model}
\label{sec:7:2}

The interaction human-mosquito is detailed
in a previous work by the authors \cite{Sofia2012b}. See also \cite{Dumont2010}.
The notation used in our mathematical model includes
four epidemiological states for humans:
\begin{description}
\item[$S_h(t)$ ---] susceptible (individuals who can contract the disease);
\item[$V_h(t)$ ---] vaccinated (individuals who were vaccinated and are now immune);
\item[$I_h(t)$ ---] infected (individuals who are capable of transmitting the disease);
\item[$R_h(t)$ ---] resistant (individuals who have acquired immunity).
\end{description}
It is assumed that the total human population $(N_h)$ is constant,
so, $N_h=S_h+V_h+I_h+R_h$. The compartment $V_h$ represents
the group of human population that is vaccinated, in order to distinguish
the resistance obtained through vaccination and the one achieved by disease recovery.
There are also three other state variables, related to the mosquitoes:
\begin{description}
\item[$A_m(t)$ ---] aquatic phase (includes the eggs, larva and pupa stages);
\item[$S_m(t)$ ---] susceptible (mosquitoes able to contract the disease);
\item[$I_m(t)$ ---] infected (mosquitoes capable of transmitting the disease to humans).
\end{description}
Similarly to the human population, it is assumed
that the total adult mosquito population is constant,
which means $N_m=S_m+I_m$. There is no resistant phase in mosquitoes due to
its short lifespan and the fact that the coefficient of disease transmission
is considered fixed. Another assumption is the susceptibility of the humans and
mosquitoes when they born. The parameters of the model are:
\begin{description}
\item[$N_h$ ---] total population;
\item[$B$ ---] average number of bites on humans by mosquitoes, per day;
\item[$\beta_{mh}$ ---] transmission probability from $I_m$ (per bite);
\item[$\beta_{hm}$ ---] transmission probability from $I_h$ (per bite);
\item[$1/\mu_{h}$ ---] average lifespan of humans (in days);
\item[$1/\eta_{h}$ ---] mean viremic period (in days);
\item[$1/\mu_{m}$ ---] average lifespan of adult mosquitoes (in days);
\item[$\varphi$ ---] number of eggs at each deposit per capita (per day);
\item[$1/\mu_{A}$ ---] natural mortality of larvae (per day);
\item[$\eta_{A}$ ---] maturation rate from larvae to adult (per day);
\item[$m$ ---] female mosquitoes per human;
\item[$k$ ---] number of larvae per human.
\end{description}
Two forms of random vaccination are possible.
The most common to reduce the prevalence of an endemic disease
is pediatric vaccination; the alternative being random vaccination
of the entire population in an outbreak.
In both cases, the vaccination can be considered perfect, conferring 100\% protection
along all life, or imperfect. This last case can be due to the difficulty
of producing an effective vaccine, the heterogeneity of the population
or even the life span of the vaccine.


\subsection{Perfect pediatric vaccine}
\label{sec:7:2:1}

For many potentially human infections, such as measles, mumps, rubella,
whooping cough, polio, there has been much focus on vaccinating newborns
or very young infants. Dengue can be a serious candidate for this type of vaccination.
In the $SVIR$ model, a continuous vaccination strategy is considered,
where a proportion of the newborn $p$ (where $0\leq p \leq 1$),
was by default vaccinated. This model also assumes that the permanent
immunity acquired through vaccination is the same as the natural immunity obtained
from infected individuals eliminating the disease naturally.
The population remains constant, \emph{i.e.}, $N_h=S_h+V_h+I_h+R_h$.
The model is represented in Figure~\ref{cap7_modeloSVIR_newborn}.
\begin{figure}[ptbh]
\begin{center}
\includegraphics[scale=0.4]{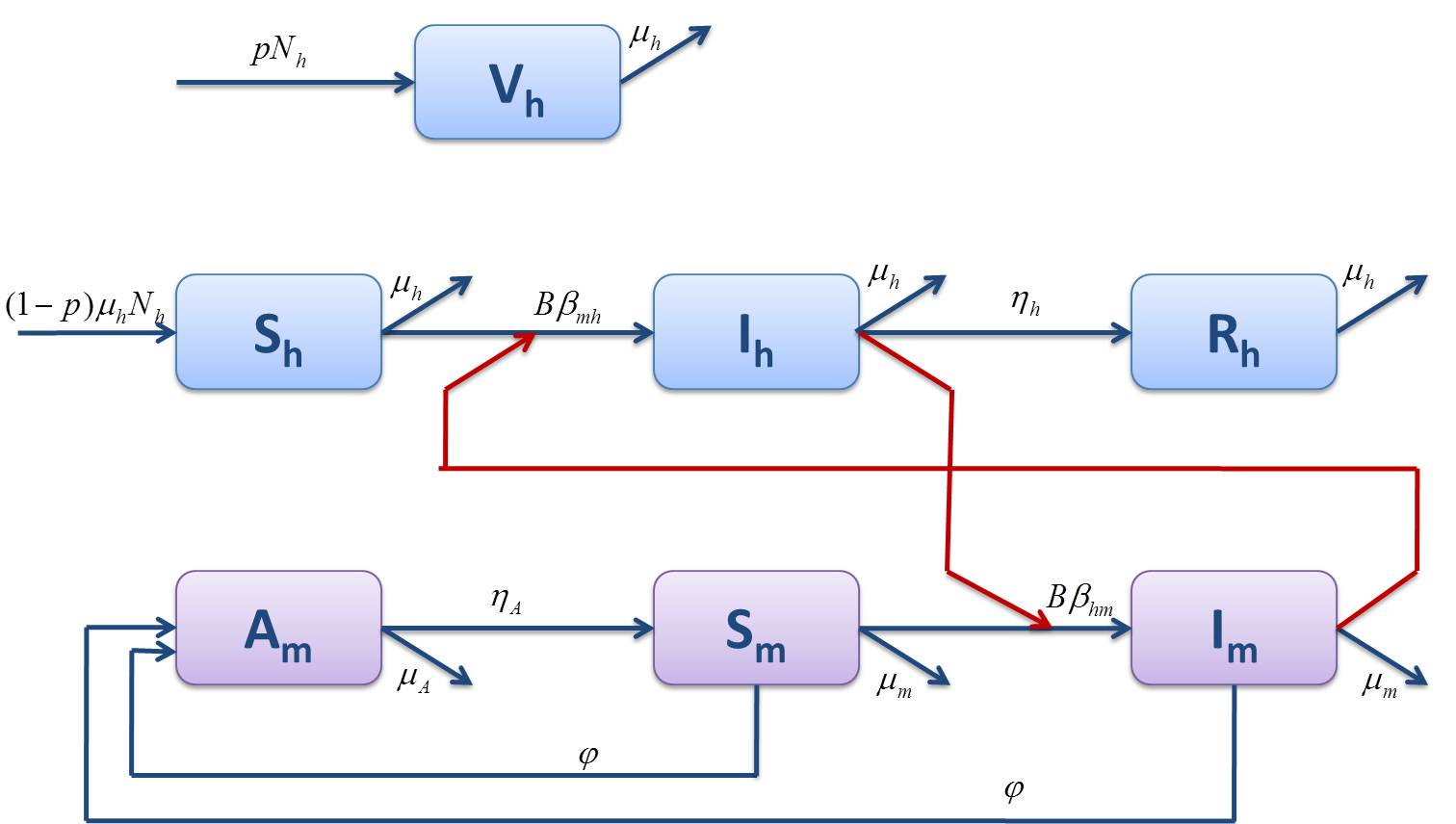}
\end{center}
\caption{Epidemic model using
a pediatric vaccine. \label{cap7_modeloSVIR_newborn}}
\end{figure}
The mathematical formulation is:
\begin{equation}
\label{cap7_ode_vaccination_newborn}
\begin{cases}
\frac{dS_h}{dt}(t) = (1-p)\mu_h N_h - \left(B\beta_{mh}\frac{I_m(t)}{N_h}+\mu_h\right)S_h(t)\\
\frac{dV_h}{dt}(t) = p\mu_h N_h-\mu_h V_h(t)\\
\frac{dI_h}{dt}(t) = B\beta_{mh}\frac{I_m(t)}{N_h}S_h(t) -(\eta_h+\mu_h) I_h(t)\\
\frac{dR_h}{dt}(t) = \eta_h I_h(t) - \mu_h R_h(t)\\
\frac{dA_m}{dt}(t) = \varphi \left(1-\frac{A_m(t)}{ k N_h}\right)(S_m(t)+I_m(t))
-\left(\eta_A+\mu_A \right) A_m(t)\\
\frac{dS_m}{dt}(t) = \eta_A A_m(t)
-\left(B \beta_{hm}\frac{I_h(t)}{N_h}+\mu_m \right) S_m(t)\\
\frac{dI_m}{dt}(t) = B \beta_{hm}\frac{I_h(t)}{N_h}S_m(t)
-\mu_m I_m(t).
\end{cases}
\end{equation}
We are assuming that the vaccine is perfect,
which means that it confers life-long protection.
The nontrivial disease-free equilibrium for system
\eqref{cap7_ode_vaccination_newborn} is given by
\begin{gather*}
S_h = (1-p) N_h,
\quad V_h = p N_h,
\quad I_h = 0,
\quad R_h = 0,\\
A_m = \left(1 - \frac{\eta_A + \mu_A}{\varphi \, \eta_A} \mu_m \right) k N_h,
\quad S_m = \frac{\eta_A}{\mu_m} A_m,
\quad I_m = 0.
\end{gather*}

As a first step, we determine the basic reproduction number without vaccination ($p=0$).

\begin{theorem}
\label{chap7_thm:r0}
Without vaccination, the basic reproduction number $\mathcal{R}_0$,
associated to the differential system \eqref{cap7_ode_vaccination_newborn},
is given by
\begin{equation}
\label{cap7:eq:R0}
\mathcal{R}_0 =  \sqrt{ \frac{k B^2 \beta_{hm} \beta_{mh}\left(
-\eta_A \mu_m-\mu_A \mu_m+\varphi \eta_A\right)}{\varphi (\eta_h + \mu_h)  \mu_m^2}}.
\end{equation}
\end{theorem}

\begin{proof}
The proof is similar to the one presented in \cite{Sofia2012}.
\end{proof}

We do all the simulations in two scenarios:
an epidemic and an endemic situation.
The following values for the parameters of the differential system
and initial conditions were used (Tables~\ref{chap7_parameters}
and \ref{chap7_initial_conditions}). These values are based on
previous works \cite{Sofia2012b,Sofia2012}.
A small number of initial infected individuals ($I_{h0} = 10$)
is considered, to simulate an early action by the health authorities. We recall
that dengue is endemic when it occurs several times in a year, and is not related with
the initial value of infected individuals. Moreover, a small initial value of infected individuals
in an endemic scenario is in agreement with \cite{Medeiros}.
In our work the same initial value of $I_h$ for both epidemic and endemic scenarios
is important in order to compare the development of the disease.
\begin{table}[ptbh]
\begin{center}
\begin{tabular}{lcc}
\hline
Parameter & Epidemic scenario & Endemic scenario\\
\hline
$N_h$ & 480000 & 480000\\
$B$ & 0.8 & 0.75\\
$\beta_{mh}$ & 0.375 & 0.21\\
$\beta_{hm}$ & 0.375 & 0.21\\
$\mu_{h}$ & $\frac{1}{71\times365}$ & $\frac{1}{71\times365}$\\
$\eta_{h}$ & $\frac{1}{3}$& $\frac{1}{3}$\\
$\mu_{m}$ & $\frac{1}{10}$ & $\frac{1}{10}$\\
$\varphi$ & 6 & 6\\
$\mu_{A}$ & $\frac{1}{4}$& $\frac{1}{4}$\\
$\eta_{A}$ & 0.08 & 0.08\\
$m$ & 3 & 3\\
$k$ & 3 & 3\\
\hline
\end{tabular}
\caption{Parameters values of the differential system \eqref{cap7_ode_vaccination_newborn}.}
\label{chap7_parameters}
\end{center}
\end{table}
\begin{table}[ptbh]
\begin{center}
\begin{tabular}{lcc}
\hline
Initial conditions & Epidemic scenario & Endemic scenario\\
\hline
$S_{h0}$ & 479990 & 379990\\
$V_{h0}$ & 0 & 0 \\
$I_{h0}$ & 10 & 10 \\
$R_{h0}$ & 0 & 100000\\
$A_{m0}$ & 1440000 & 1440000\\
$S_{m0}$ & 1440000 & 1440000\\
$I_{m0}$ & 0 & 0\\
\hline
\end{tabular}
\caption{Initial conditions of the differential system \eqref{cap7_ode_vaccination_newborn}.}
\label{chap7_initial_conditions}
\end{center}
\end{table}

Two main differences between an epidemic episode and an endemic situation were found.
Firstly, in the endemic situation there was a slight decrease in the average daily
biting $B$ and transmission probabilities
$\beta_{mh}$ and $\beta_{hm}$, which can be explained by the fact that the mosquito
may have more difficulties to find a naive individual.
The second difference is concerned with the strong increase of the initial human
population that is resistant to the disease. This can be explained by the fact
that the disease, in an endemic situation, already creates an immune resistance
to the infection, \emph{i.e.}, the population already has herd immunity.
With our values we obtain approximately $\mathcal{R}_{0}=2.46$ and $\mathcal{R}_{0}=1.29$
for epidemic and endemic scenarios, respectively.

During an outbreak, the disease transmission assumes different behaviors,
according to the distinct scenarios, as can been seen in Figure~\ref{cap7_population_no_vaccine}.
In one year, the peak in an epidemic situation could reach more than 80000 cases.
In contrast, in the endemic situation the curve of infected individuals has a more smooth behavior
and reaches a peak less than 3000 cases. Figure~\ref{cap7_population_no_vaccine_mosquito}
relates to the mosquito population. In the endemic scenario, because a substantial part
of the human population is resistant to the disease, the infected mosquitoes bite
a considerable percentage of resistant host and, as consequence,
the disease is not transmitted.
\begin{figure}
\centering
\subfloat[Epidemic scenario]{\label{cap7_epid_no_vaccine}
\includegraphics[scale=0.52]{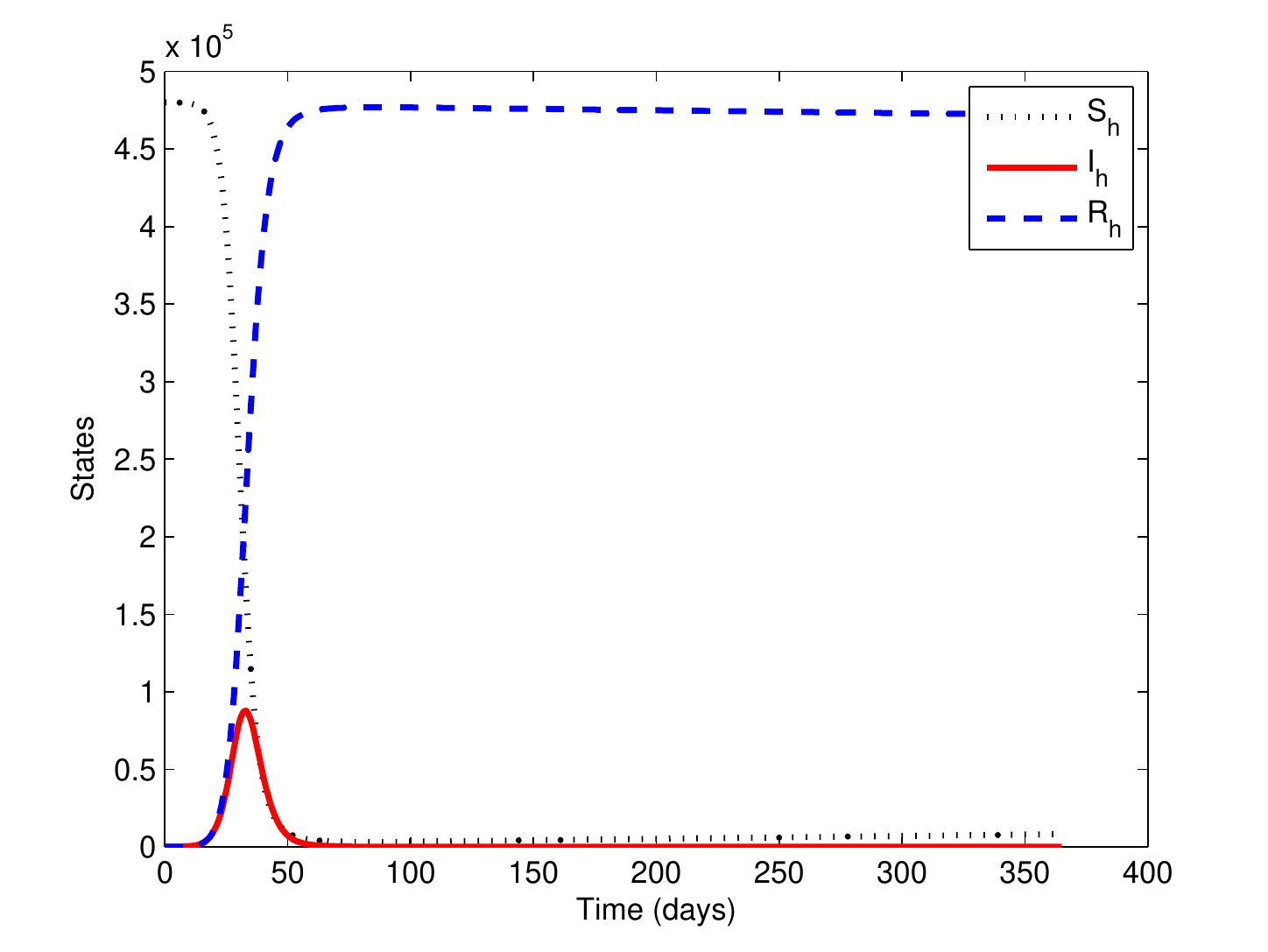}}
\subfloat[Endemic scenario]{\label{cap7_end_no_vaccine}
\includegraphics[scale=0.52]{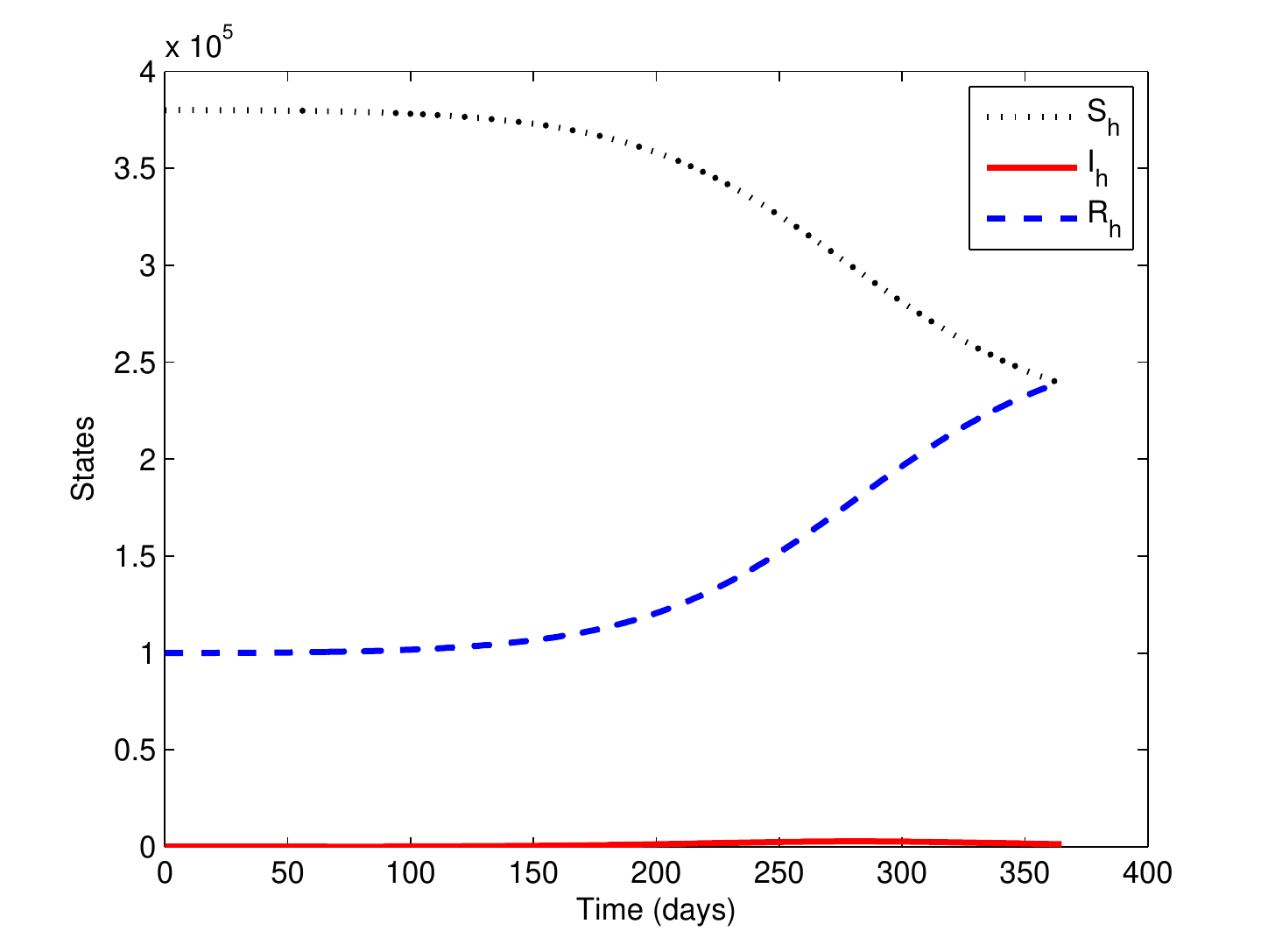}}
\caption{Human population in a dengue outbreak, without vaccine.}
\label{cap7_population_no_vaccine}
\end{figure}
\begin{figure}
\centering
\subfloat[Epidemic scenario]{\label{cap7_epid_no_vaccine_mosquito}
\includegraphics[scale=0.52]{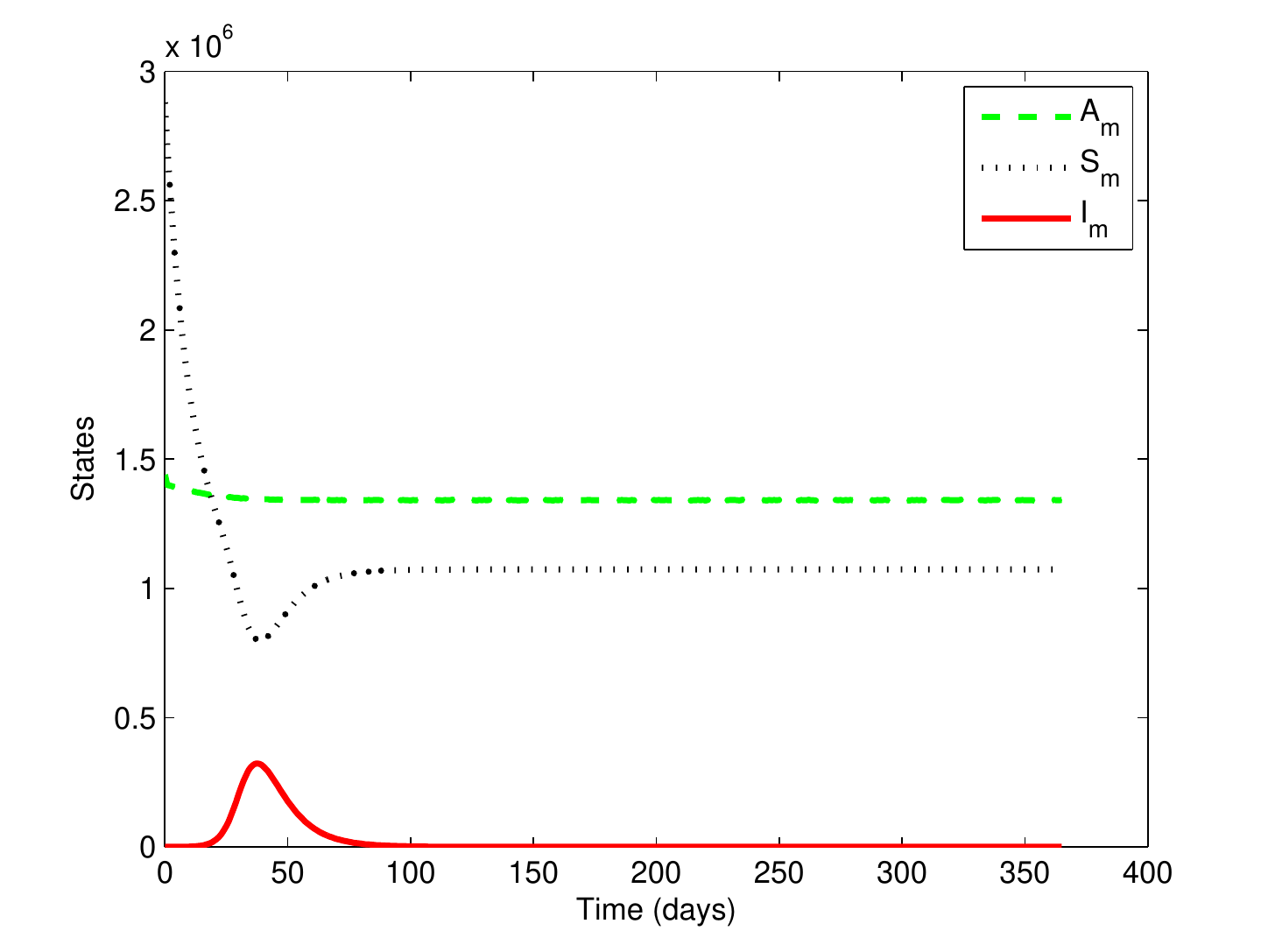}}
\subfloat[Endemic scenario]{\label{cap7_end_no_vaccine_mosquito}
\includegraphics[scale=0.52]{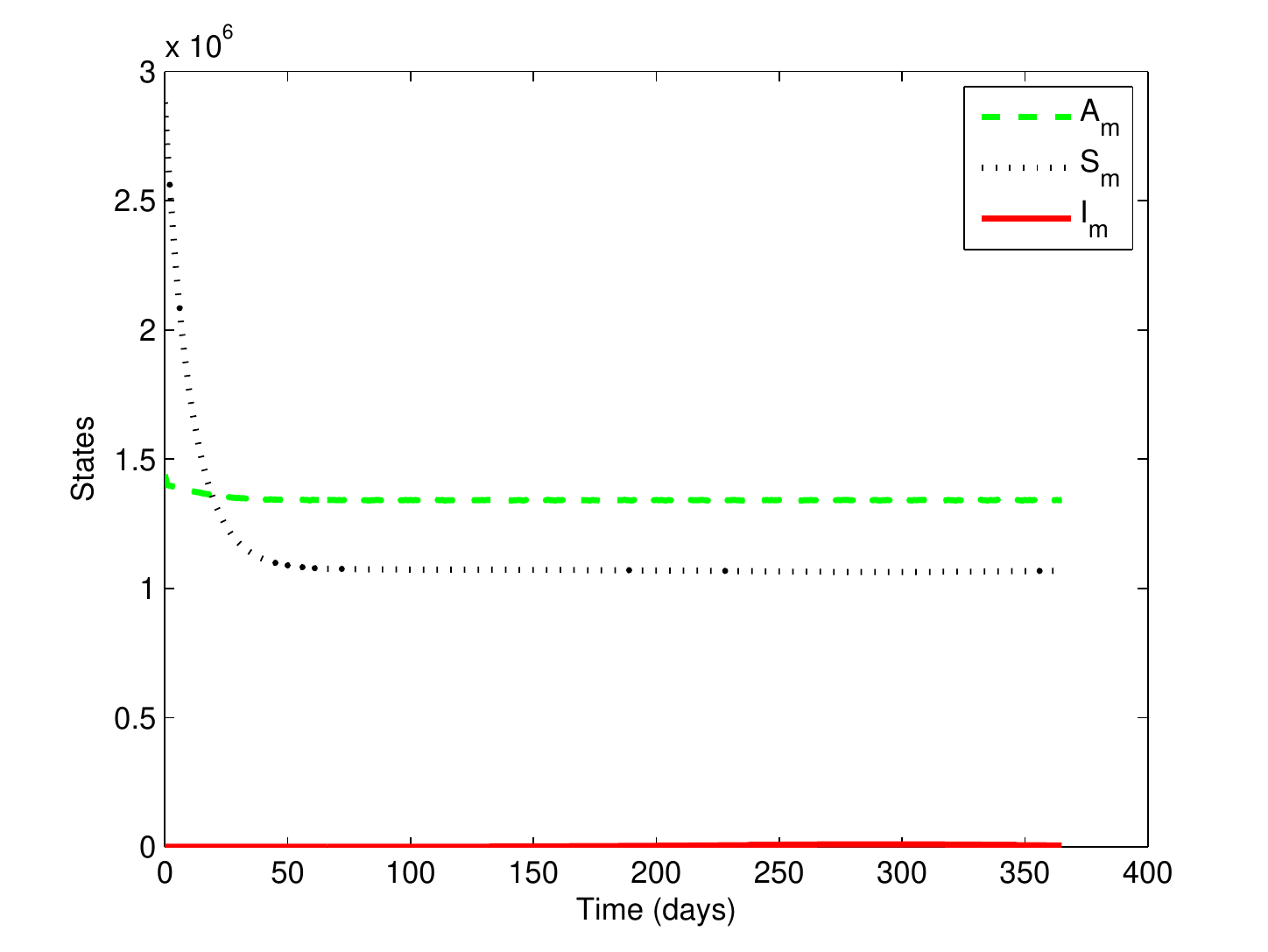}}
\caption{Mosquito population in a dengue outbreak, without vaccine.}
\label{cap7_population_no_vaccine_mosquito}
\end{figure}

Suppose that at time $t=0$ a proportion $p$ of newborns is vaccinated
with a perfect vaccine without side effects. Since this proportion,
$p$, is now immune, $\mathcal{R}_{0}$ is reduced,
creating a new basic reproduction number $\mathcal{R}_0^{p}$.

\begin{definition}[\protect{cf. \cite{Scherer2002}}]
\label{chap7_thm:r0:newborn}
The basic reproduction number with pediatric vaccination,
$\mathcal{R}_0^{p}$, associated to the differential system
\eqref{cap7_ode_vaccination_newborn}, is given by
$\mathcal{R}_0^{p} = (1-p)\mathcal{R}_{0}$,
where $\mathcal{R}_{0}$ is defined by \eqref{cap7:eq:R0}.
\end{definition}

Observe that $\mathcal{R}_0^{p}\leq\mathcal{R}_0$.
Equality is only achieved when $p=0$,
\emph{i.e.}, when there is no vaccination.
Note that the constraint $\mathcal{R}_0^{p}<1$ defines implicitly
a critical vaccination portion $p > p_{c} := 1-\frac{1}{\mathcal{R}_{0}}$
that must be achieved for eradication.
Since vaccination entails costs, to choose the smallest coverage
that achieves eradication is the best option. This way, the entire
population does not need to be vaccinated in order to eradicate the disease
(this is the herd immunity phenomenon).
Vaccinating at the critical level $p_{c}$ does not instantly lead
to disease eradication. The immunity level within the population
requires time to build up and at the critical level it may take
a few generations before the required herd immunity is achieved.
Thus, from a public health perspective, $p_{c}$ acts as a lower bound
on what should be achieved, with higher levels of vaccination
leading to a more rapid elimination of the disease.
Figure~\ref{cap7_p_variation} shows simulations
with different proportions of the newborns vaccinated,
in both epidemic and endemic scenarios. Note that at time $t=0$ no person was vaccinated.
In the epidemic situation, as the outbreak reaches a peak at the beginning of the year,
the proportion of newborns vaccinated at that time is minimum
and cannot influence the curve of infected individuals,
giving the optical illusion of a single curve. On the other hand,
in the endemic case, as the outbreak occurs later, the vaccination campaign
starts to produce effects, decreasing the total number of sick humans. This last
graphic illustrates that a vaccination campaign centered in newborns
is a bet for the future of a country, but does not produce instantly results
to fight the disease. To achieve immediate results, it is necessary
to use random mass vaccination, which means that it is necessary
to vaccine a significant part of the population.
\begin{figure}
\centering
\subfloat[Epidemic scenario]{\label{cap7_epid_p_variation}
\includegraphics[scale=0.53]{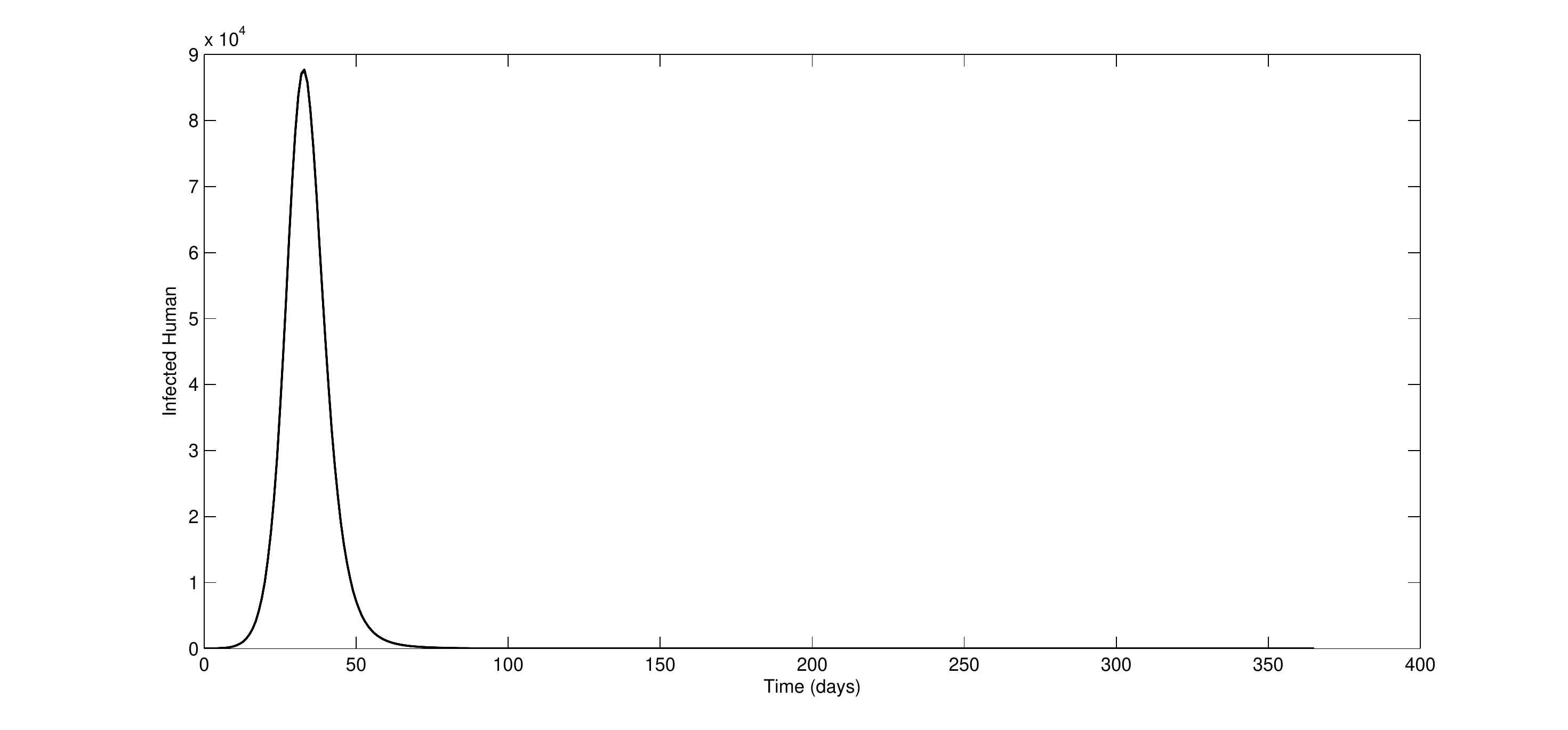}}\\
\subfloat[Endemic scenario]{\label{cap7_end_p_variation}
\includegraphics[scale=0.53]{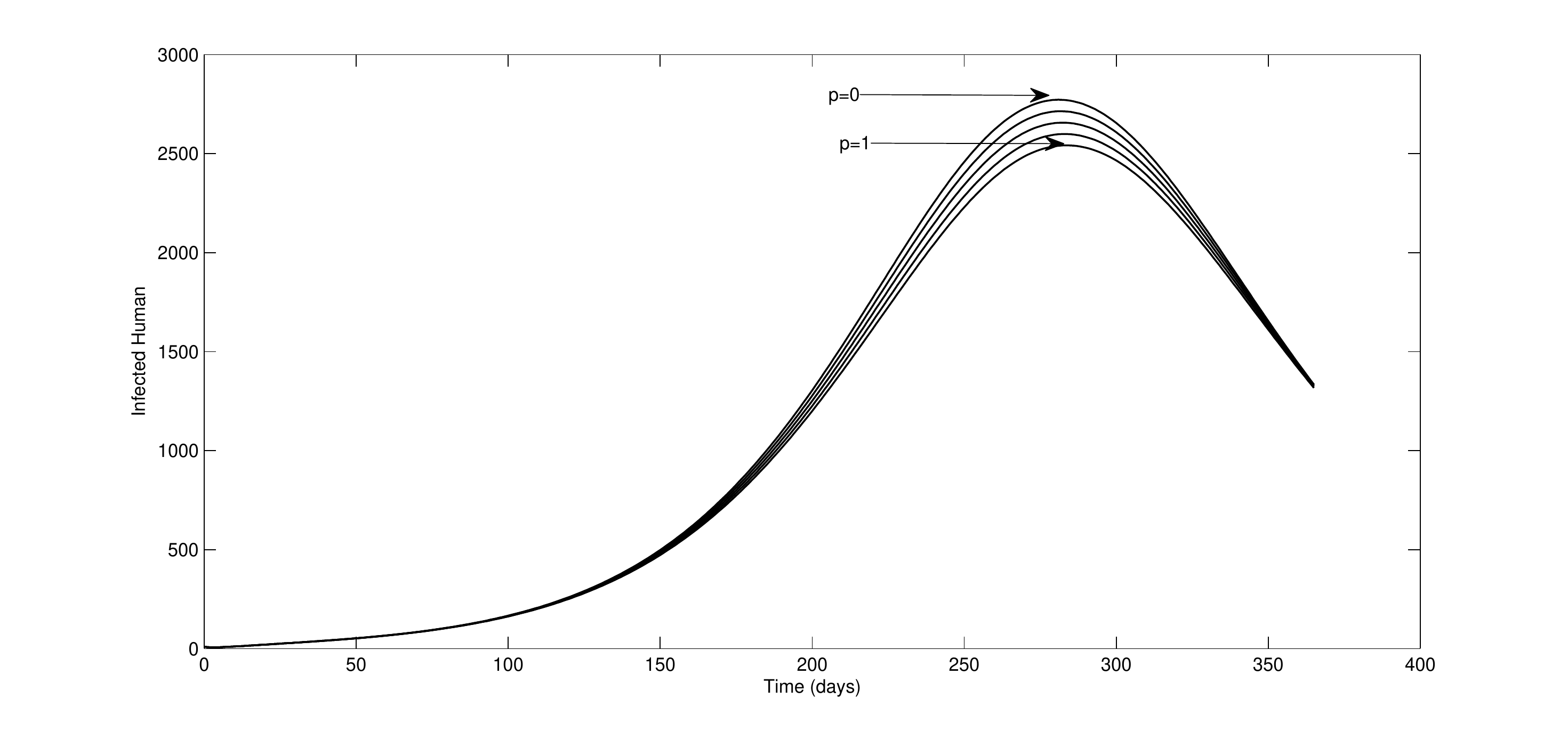}}
\caption{Infected human in an outbreak, varying the proportion
of newborns vaccinated ($p=0, 0.25, 0.50, 0.75, 1$).}
\label{cap7_p_variation}
\end{figure}


\subsection{Perfect random mass vaccination}
\label{sec:7:2:2}

A mass vaccination program may be initiated whenever there is an increase of
the risk of an epidemic. In such situations, there is a competition between
the exponential increase of the epidemic and the logistical constraints
upon mass vaccination. For most human diseases it is possible, and more efficient,
to not vaccinate those individuals who have recovered from the disease, because
they are already protected. Another situation could be the introduction
of a new vaccine in a population that lives an endemic situation.
Let us consider the control technique of constant vaccination of susceptibles.
In this scheme a fraction $0 \leq \psi \leq 1$ of the entire susceptible population,
not just newborns, is being continuously vaccinated. It is assumed that the permanent
immunity acquired by vaccination is the same as natural immunity obtained
from infected individuals in recovery. The epidemiological scheme
is presented in Figure~\ref{cap7_modeloSVIR_mass}.
\begin{figure}[ptbh]
\begin{center}
\includegraphics[scale=0.4]{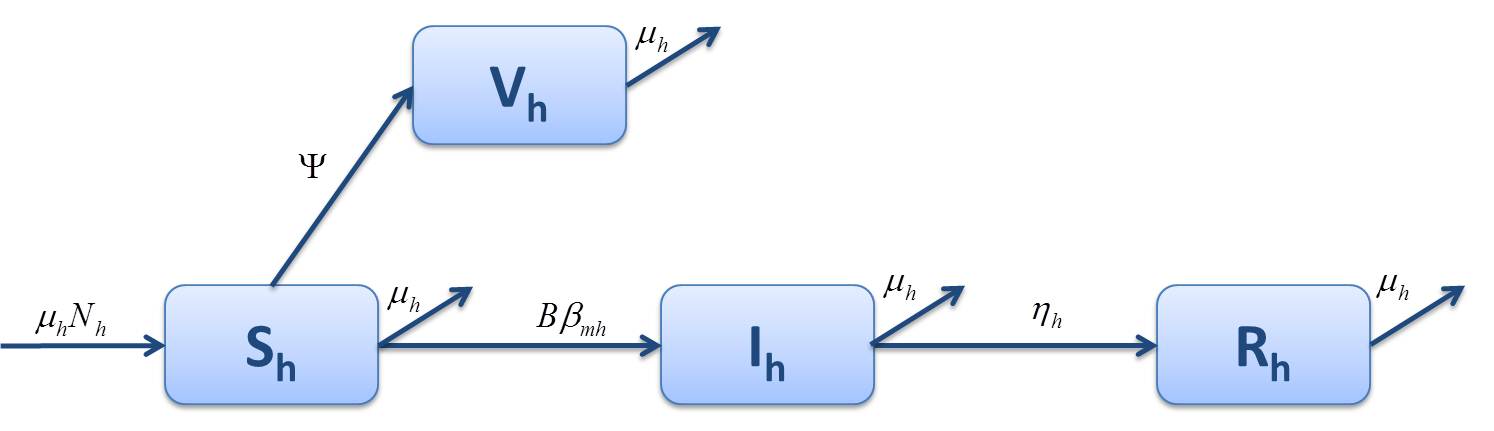}
\end{center}
\caption{Epidemic model for human population using
a mass random vaccine. \label{cap7_modeloSVIR_mass}}
\end{figure}
The mosquito population remains equal to the previous subsection,
while the mathematical formulation for human population in the
new epidemiological scheme is given by
\begin{equation}
\label{cap7_ode_vaccination_mass}
\begin{cases}
\frac{dS_h}{dt}(t) = \mu_h N_h - \left(B\beta_{mh}\frac{I_m(t)}{N_h}+\psi+\mu_h\right)S_h(t)\\
\frac{dV_h}{dt}(t) = \psi S_h(t)-\mu_h V_h(t)\\
\frac{dI_h}{dt}(t) = B\beta_{mh}\frac{I_m(t)}{N_h}S_h(t) -(\eta_h+\mu_h) I_h(t)\\
\frac{dR_h}{dt}(t) = \eta_h I_h(t) - \mu_h R_h(t).
\end{cases}
\end{equation}
For this model, we define a new basic reproduction number $\mathcal{R}_0^{\psi}$.

\begin{definition}[cf. \cite{Zhou2003}]
\label{chap7_thm:r0:mass}
The basic reproduction number with random mass vaccination,
$\mathcal{R}_0^{\psi}$, associated to the differential
system \eqref{cap7_ode_vaccination_mass}, is given by
\begin{equation}
\label{eq:R0_mass}
\mathcal{R}_0^{\psi} = \mathcal{R}_{0}\left(\frac{\mu_h}{\mu_h+\psi}\right),
\end{equation}
where $\mathcal{R}_{0}$ is defined by \eqref{cap7:eq:R0}.
\end{definition}

Comparing this model with the model of constant vaccination of newborns,
it is apparent that instead of constantly vaccinating a portion of newborns,
a part of the entire susceptible population is now being continuously
vaccinated. Since the natural birth rate $\mu_h$ is usually small, the fraction $p\mu_h$
of newborns being continuously vaccinated will be also small, whereas in this
model, a larger group $\psi S_h$ of susceptible can be continuously vaccinated.
For this reason, we expect this model to require
a smaller proportion $\psi$ to achieve eradication.
Note that $\mathcal{R}_0^{\psi}\leq\mathcal{R}_0$. Equality is only achieved
in the limit, when $\psi=0$, that is, when there is no vaccination.
The constraint $\mathcal{R}_0^{\psi}<1$ defines implicitly a critical
vaccination portion $\psi > \psi_{c}$ that must be achieved for eradication:
$$
\psi_{c}=\left(\mathcal{R}_{0}-1\right)\mu_h.
$$
Figure~\ref{cap7_psi_variation} illustrates the variation of the number
of infected people when a mass vaccination is introduced.
The graphs present five simulations using different proportions
of the susceptible being vaccinated:
$\psi=0.05, 0.10, 0.25, 0.50, 1$. Observe that,
in spite of the calculations being done in the period of 365 days, the figures
only show suitable windows, in order to provide a better analysis.
In both epidemic and endemic scenarios, even with a small coverage of the population,
vaccination dramatically decreases the number of infected.
In the epidemic scenario, the situation has changed from around 80000 cases
(with no vaccination, Figure~\ref{cap7_population_no_vaccine})
to less than 1200 cases, vaccinating only 5\% of the susceptible.
In the endemic scenario, the decrease is even more accentuated.
\begin{figure}
\centering
\subfloat[Epidemic scenario]{\label{cap7_epid_psi_variation}
\includegraphics[scale=0.53]{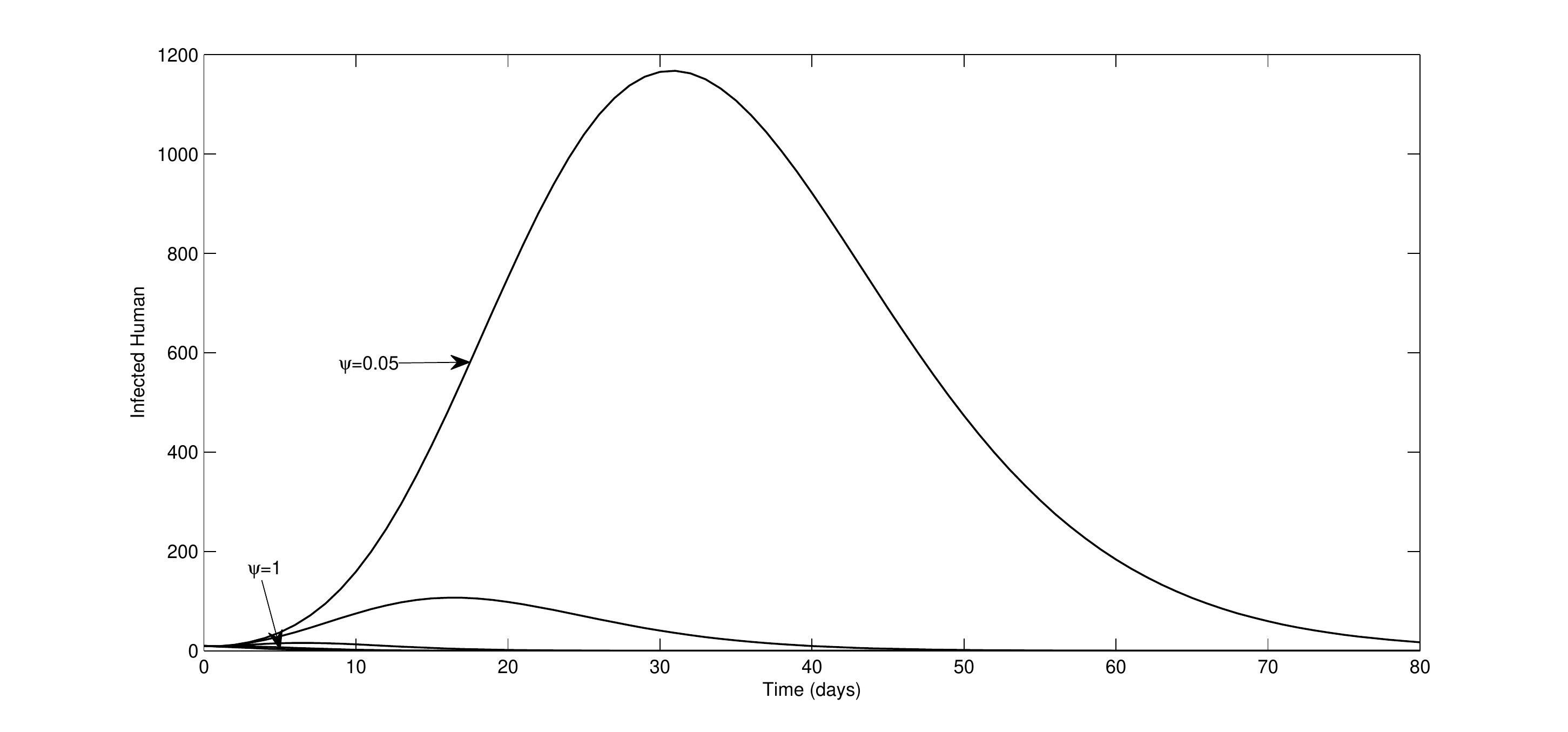}}\\
\subfloat[Endemic scenario]{\label{cap7_end_psi_variation}
\includegraphics[scale=0.53]{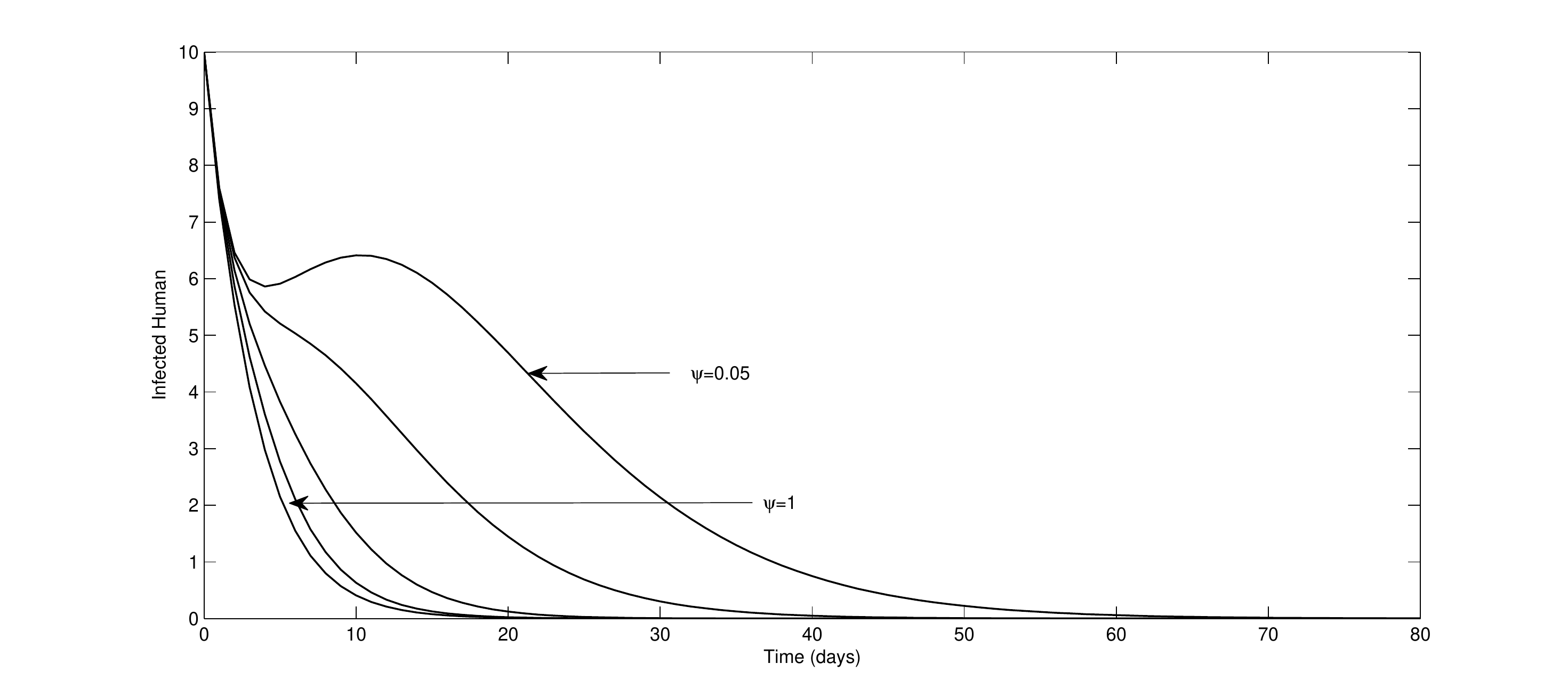}}
\caption{Infected human in an outbreak, varying the proportion
of susceptible population vaccinated ($\psi=0.05, 0.10, 0.25, 0.50, 1$).}
\label{cap7_psi_variation}
\end{figure}

Until here, we have considered a perfect vaccine, which means that
every vaccinated individual remains resistant to the disease.
However, a majority of the available vaccines for the human population
does not produce 100\% success in the disease battle. Usually,
the vaccines are imperfect, which means that a minor percentage of cases,
in spite of vaccination, are infected.


\subsection{Imperfect random mass vaccination}
\label{sec:7:2:3}

Most of the theory about disease evolution is based on the assumption
that the host population is homogeneous. Individual hosts, however,
may differ and they may constitute very different habitats. In particular,
some habitats may provide more resources or be more vulnerable
to virus exploitation \cite{Gandon2003}. The use of models with
imperfect vaccines can describe better this type of human heterogeneity.
Another explanation for the use of imperfect vaccines is that until now
we had considered models assuming that as soon as individuals
begin the vaccination process, they become immediately immune to the disease.
However, the time it takes for individuals to obtain immunity by completing
a vaccination process cannot be ignored, because meanwhile an individual can be infected.
In this section a continuous vaccination strategy is considered,
where a fraction $\psi$ of the susceptible class was vaccinated.
The vaccination may reduce but not completely eliminate susceptibility to infection.
For this reason, we consider a factor $\sigma$ as the infection rate of
vaccinated members. When $\sigma=0$, the vaccine is perfectly
effective, when $\sigma=1$, the vaccine has no effect at all.
The value $1-\sigma$ can be understood as the efficacy level of the vaccine.
The new model for the human population is represented
in Figure~\ref{cap7_modeloSVIR_imperfect}.
\begin{figure}[ptbh]
\center
\includegraphics[scale=0.4]{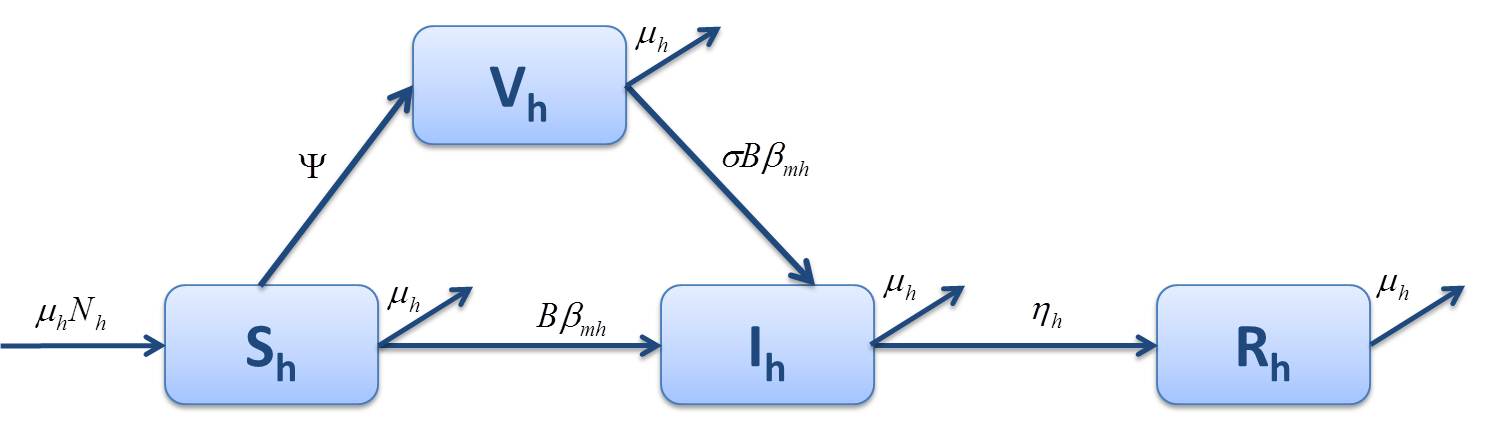}
\caption{\label{cap7_modeloSVIR_imperfect} Epidemiological $SVIR$ model
for human population with an imperfect vaccine.}
\end{figure}
Accordingly, we have the following system of differential equations:
\begin{equation}
\label{cap7_ode_vaccination_imperfect}
\begin{cases}
\frac{dS_h}{dt} = \mu_h N_h  - \left(B\beta_{mh}\frac{I_m}{N_h}+\psi +\mu_h\right)S_h\\
\frac{dV_h}{dt} = \psi S_h-\left(\sigma B \beta_{mh}\frac{I_m}{N_h}+\mu_h\right)V_h\\
\frac{dI_h}{dt} = B\beta_{mh}\frac{I_m}{N_h}(S_h+\sigma V_h) -(\eta_h+\mu_h) I_h\\
\frac{dR_h}{dt} = \eta_h I_h - \mu_h R_h.
\end{cases}
\end{equation}
As expected, a new basic reproduction number is associated to
\eqref{cap7_ode_vaccination_imperfect}.

\begin{definition}[cf. \cite{Liu2008}]
\label{chap7_thm:r0:imperfect}
The basic reproduction number with an imperfect vaccine, $\mathcal{R}_0^{\sigma}$,
associated to the differential system \eqref{cap7_ode_vaccination_imperfect},
is defined by
\begin{equation*}
\label{eq:R0_imperfect}
\mathcal{R}_0^{\sigma} = \left(1+\sigma \psi\right)\frac{\mu_h}{\mu_h
+\psi}\mathcal{R}_{0}=\left(1+\sigma \psi\right)\mathcal{R}_0^{\psi},
\end{equation*}
where $\mathcal{R}_0^{\psi}$ is given by \eqref{eq:R0_mass}.
\end{definition}

Note that $\mathcal{R}_0^{\psi}\leq\mathcal{R}_0^{\sigma}$ and when
the vaccine is perfect, i.e., $\sigma=0$, $\mathcal{R}_0^{\sigma}$ degenerates
into $\mathcal{R}_0^{\psi}$. In other words, a high efficacy vaccine
leads to a lower vaccination coverage to eradicate the disease.
However, it is realized in \cite{Liu2008} that it is much more difficult to increase the
efficacy level of the vaccine when compared to controlling the vaccination rate $\psi$.
Figures~\ref{cap7_sigma_variation:A} and \ref{cap7_sigma_variation:B} show several simulations,
by varying the vaccine efficacy and the percentage of population that is vaccinated.
Comparing with Figure~\ref{cap7_epid_psi_variation}, in the epidemic scenario
with a perfect vaccine, the number of human infected has reached to a maximum peak
of 1200 cases per day, in the worst scenario ($\psi=0.05$). Using an imperfect vaccine,
with a level of efficacy of 80\% (Figure~\ref{epid_psi_vary}),
with the same values for $\psi$, the maximum peak increases until 9000 cases.
We conclude that the production of a vaccine with a high level
of efficacy has a preponderant role in the reduction of the disease spread.
Figures~\ref{epid_sigma_vary} and \ref{end_sigma_vary} reinforce the previous sentence.
Assuming that 85\% of the population is vaccinated, the number of infected cases
decreases sharply with the increasing of the  effectiveness level of the vaccine.
\begin{figure}
\centering
\subfloat[Epidemic scenario]{\label{epid_psi_vary}
\includegraphics[scale=0.53]{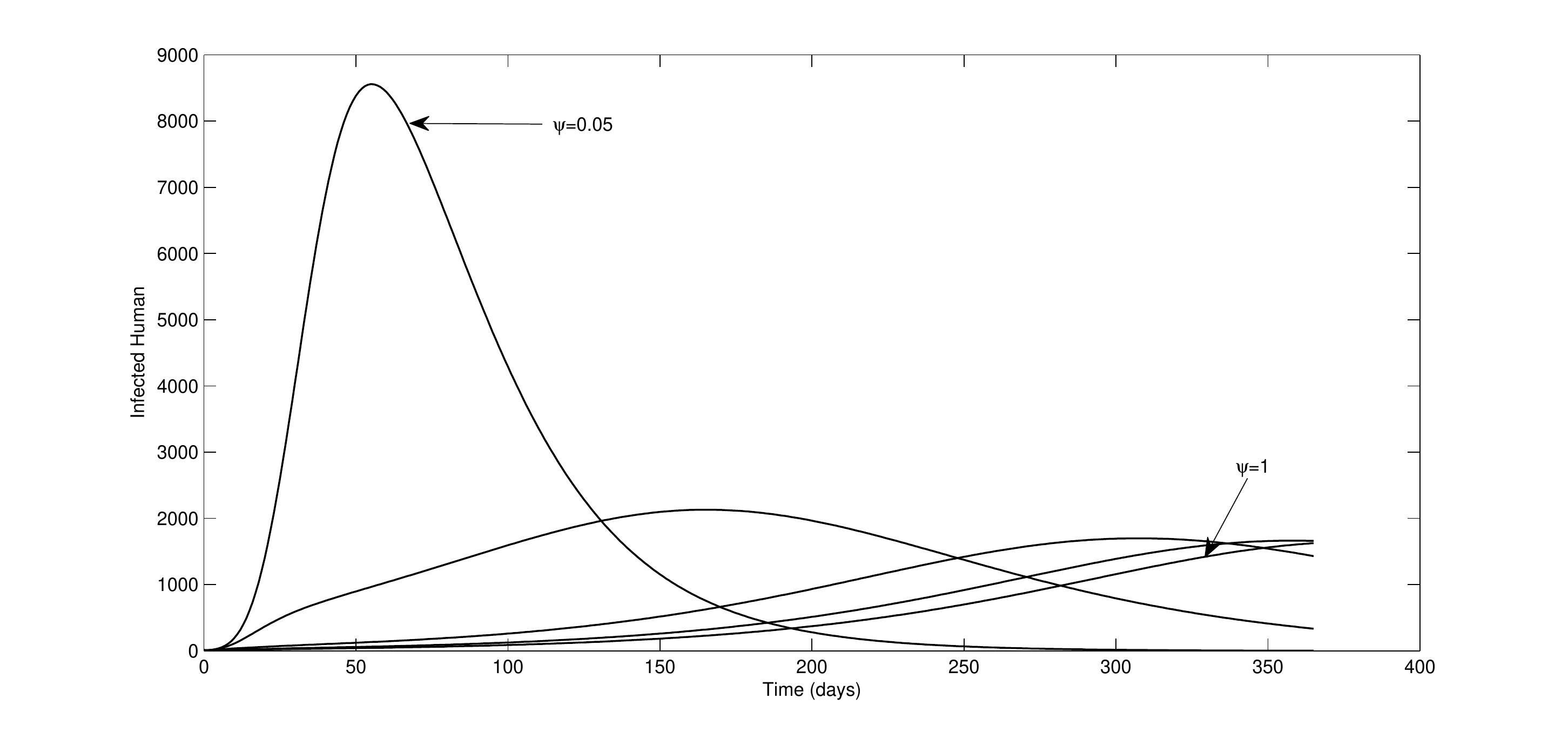}}\\
\subfloat[Endemic scenario]{\label{end_psi_vary}
\includegraphics[scale=0.53]{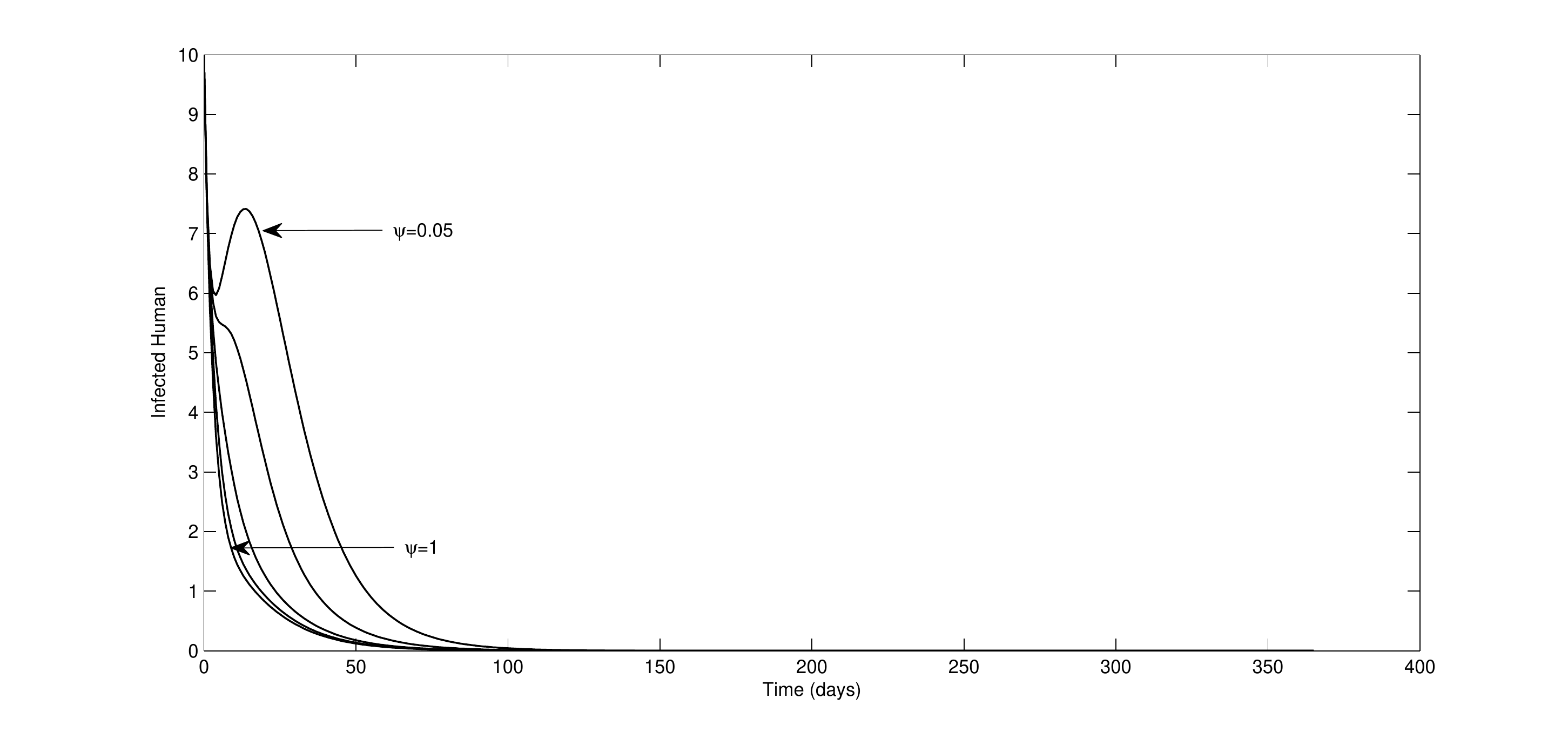}}
\caption{Infected humans in an outbreak varying
the proportion of susceptible population vaccinated ($\psi=0.05, 0.10, 0.25, 0.50, 1$)
with a vaccine simulating 80\% of effectiveness ($\sigma=0.2$)}
\label{cap7_sigma_variation:A}
\end{figure}
\begin{figure}
\centering
\subfloat[Epidemic scenario]{\label{epid_sigma_vary}
\includegraphics[scale=0.53]{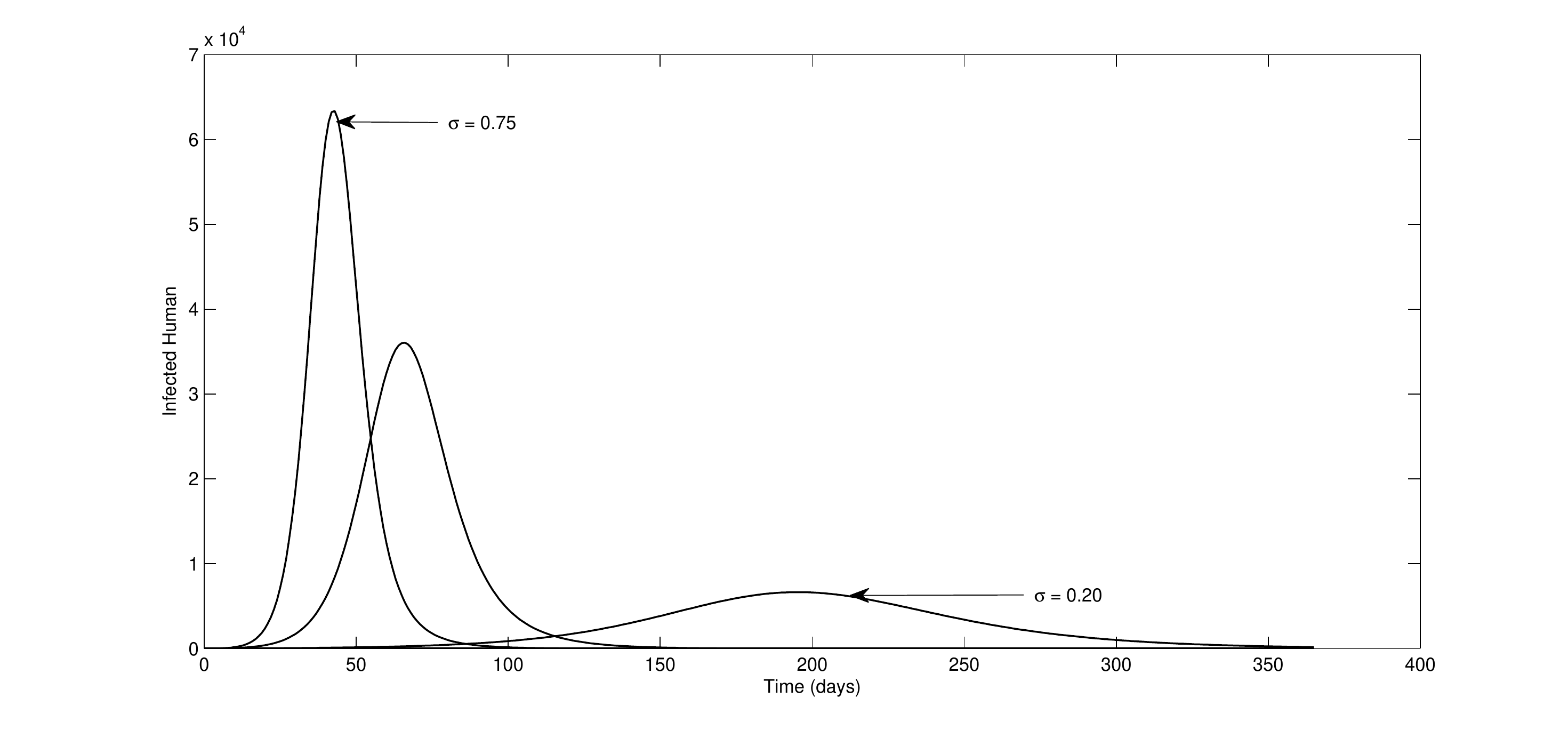}}\\
\subfloat[Endemic scenario]{\label{end_sigma_vary}
\includegraphics[scale=0.53]{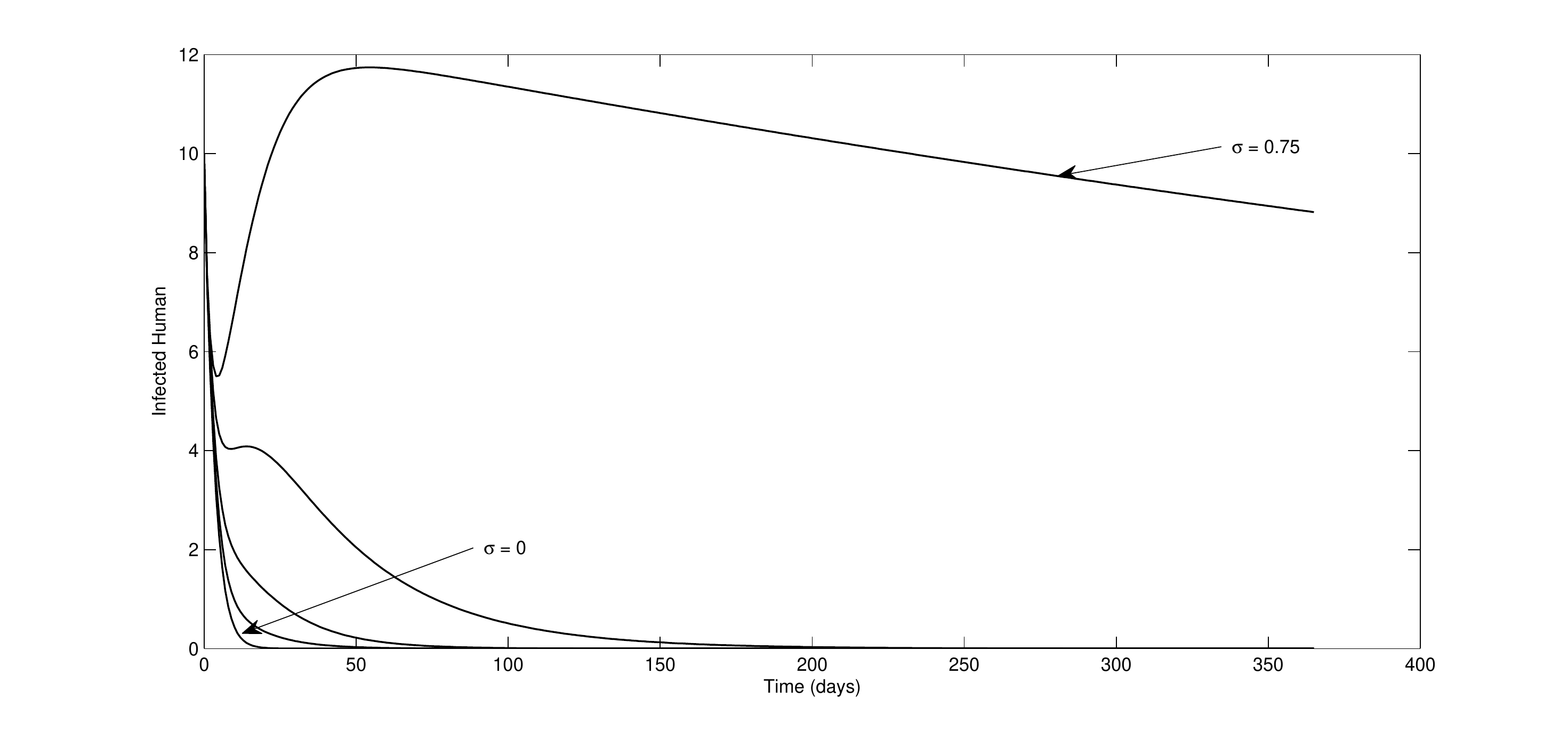}}
\caption{Infected humans in an outbreak varying the efficacy
level of the vaccine ($\sigma=0, 0.10, 0.20, 0.50, 0.75$)
and considering that 85\% of the human susceptible
population is vaccinated ($\psi=0.85$).}
\label{cap7_sigma_variation:B}
\end{figure}
According to \cite{DeRoeck2003}, an acceptable level of efficacy is
at least 80\%  against all four serotypes, and 3 to 5 for the
length of protection. These values are commonly considered, across countries,
as the minimum acceptable levels.

Next we study another type of imperfect vaccine:
one that confers a limited life-long protection.


\subsection{Random mass vaccination with waning immunity}
\label{sec:7:2:4}

Until the 1990s, the universal assumption
of mathematical models of vaccination was: there is no
waning of vaccine-induced immunity. This assumption was
routinely made because, for most of the major vaccines against childhood
infectious diseases, it is approximately correct \cite{Scherer2002}.

Suppose that the immunity, obtained by the vaccination process, is temporary.
Assume that immunity has the waning rate $\theta$. Then the model for humans is given by
\begin{equation}
\label{cap7_ode_vaccination_waning}
\begin{cases}
\frac{dS_h}{dt}(t) = \mu_h N_h  +\theta V_h(t)
- \left(B\beta_{mh}\frac{I_m(t)}{N_h}+\psi +\mu_h\right)S_h(t)\\
\frac{dV_h}{dt}(t) = \psi S_h(t)-\left(\theta +\mu_h\right)V_h(t)\\
\frac{dI_h}{dt}(t) = B\beta_{mh}\frac{I_m(t)}{N_h}S_h(t) -(\eta_h+\mu_h) I_h(t)\\
\frac{dR_h}{dt}(t) = \eta_h I_h(t) - \mu_h R_h(t).
\end{cases}
\end{equation}
This model can be represented by the epidemiological scheme of
Figure~\ref{cap7_modeloSVIR_waning}.
\begin{figure}[ptbh]
\center
\includegraphics[scale=0.4]{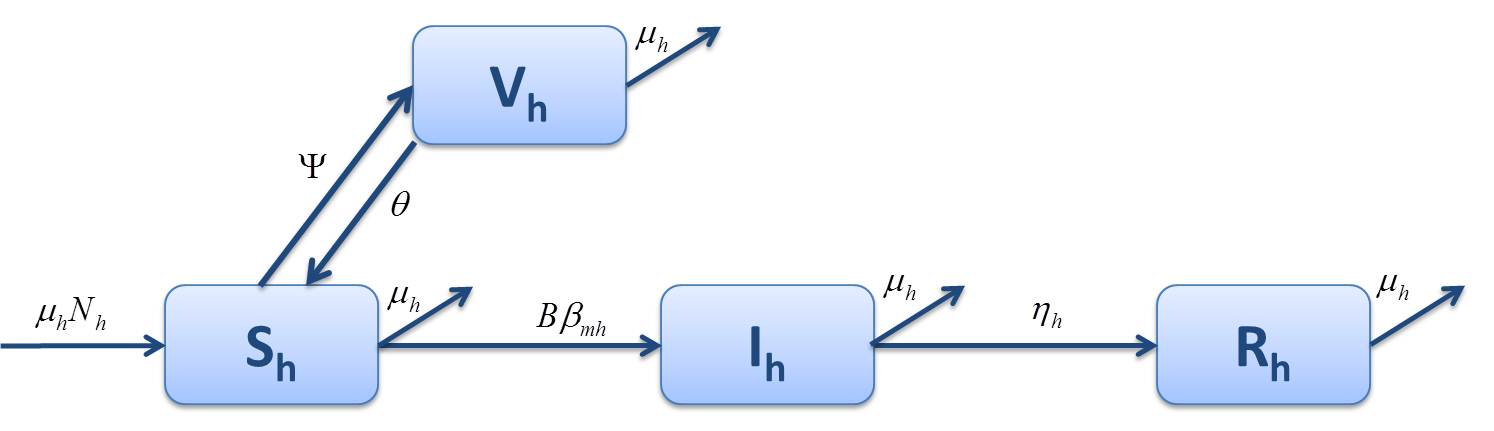}
\caption{\label{cap7_modeloSVIR_waning} Epidemiological $SVIR$ model
for human population with a waning immunity vaccine.}
\end{figure}
This leads naturally to the following basic reproduction number.

\begin{definition}
\label{chap7_thm:r0:waning}
The basic reproduction number with an imperfect vaccine,
$\mathcal{R}_0^{\theta}$, associated to the differential
system \eqref{cap7_ode_vaccination_waning}, is defined by
$\mathcal{R}_0^{\theta} = \mathcal{R}_0^{\psi}$,
where $\mathcal{R}_{0}^{\psi}$ is given by \eqref{eq:R0_mass}.
\end{definition}

According to \cite{Stechlinski2009}, the basic reproduction numbers
$\mathcal{R}_0^{\theta}$ and $\mathcal{R}_0^{\psi}$ are the same,
because the disease will still spread at the same rate with or without temporary immunity.
However, we should expect that the convergence rate will be different
between the random mass vaccination and random mass vaccination with waning immunity,
since the disease will be eradicated faster in the constant treatment model without
waning immunity compared to the other with waning immunity.
Figure~\ref{cap7_theta_variation} illustrates this statement.
Considering that 85\% of human population is vaccinated, the number of infected
is increasing as the value of the waning immunity is growing
($\theta=0, 0.05, 0.10, 0.15, 0.20$).
\begin{figure}
\centering
\subfloat[Epidemic scenario]{\label{cap7_epid_theta_variation}
\includegraphics[scale=0.53]{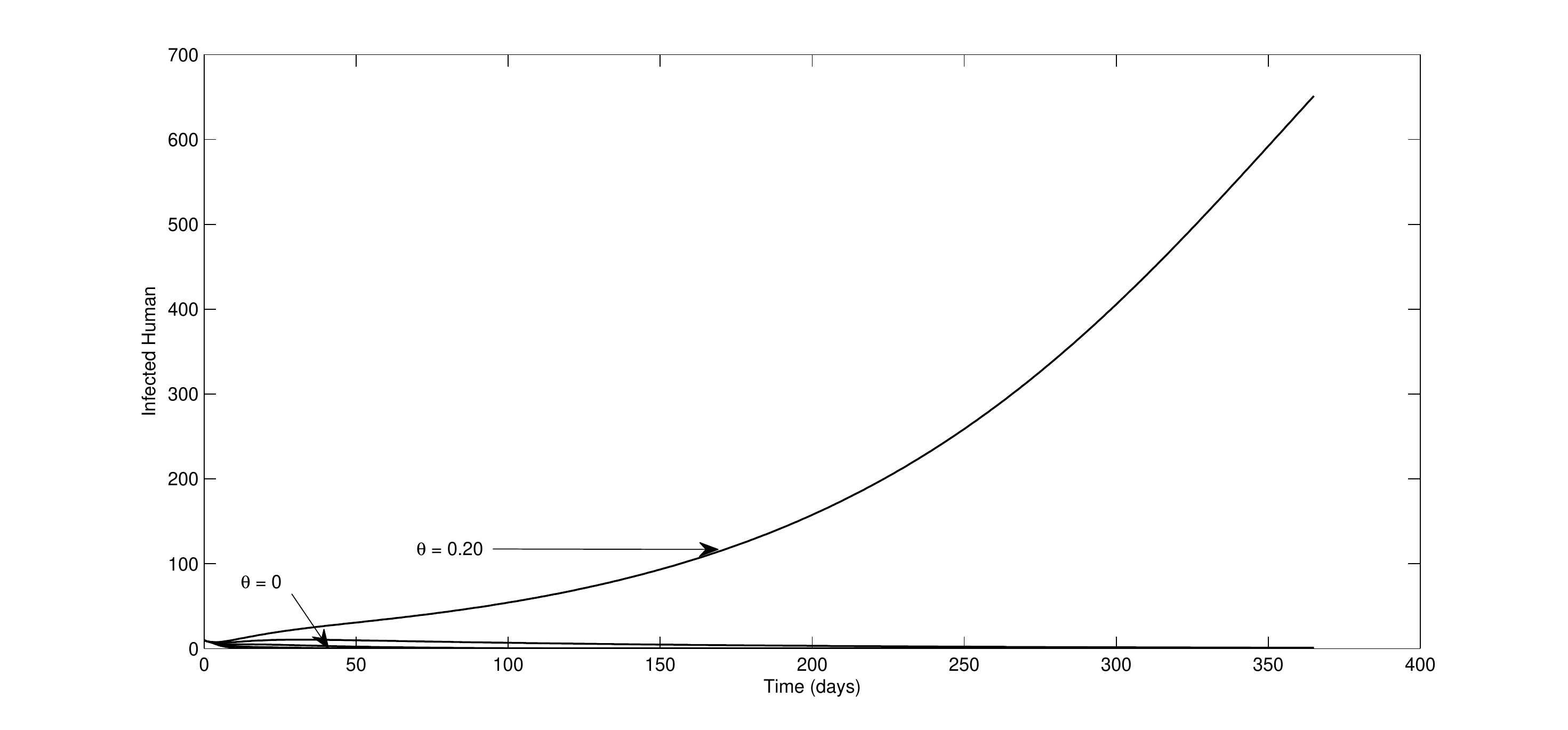}}\\
\subfloat[Endemic scenario]{\label{cap7_end_theta_variation}
\includegraphics[scale=0.53]{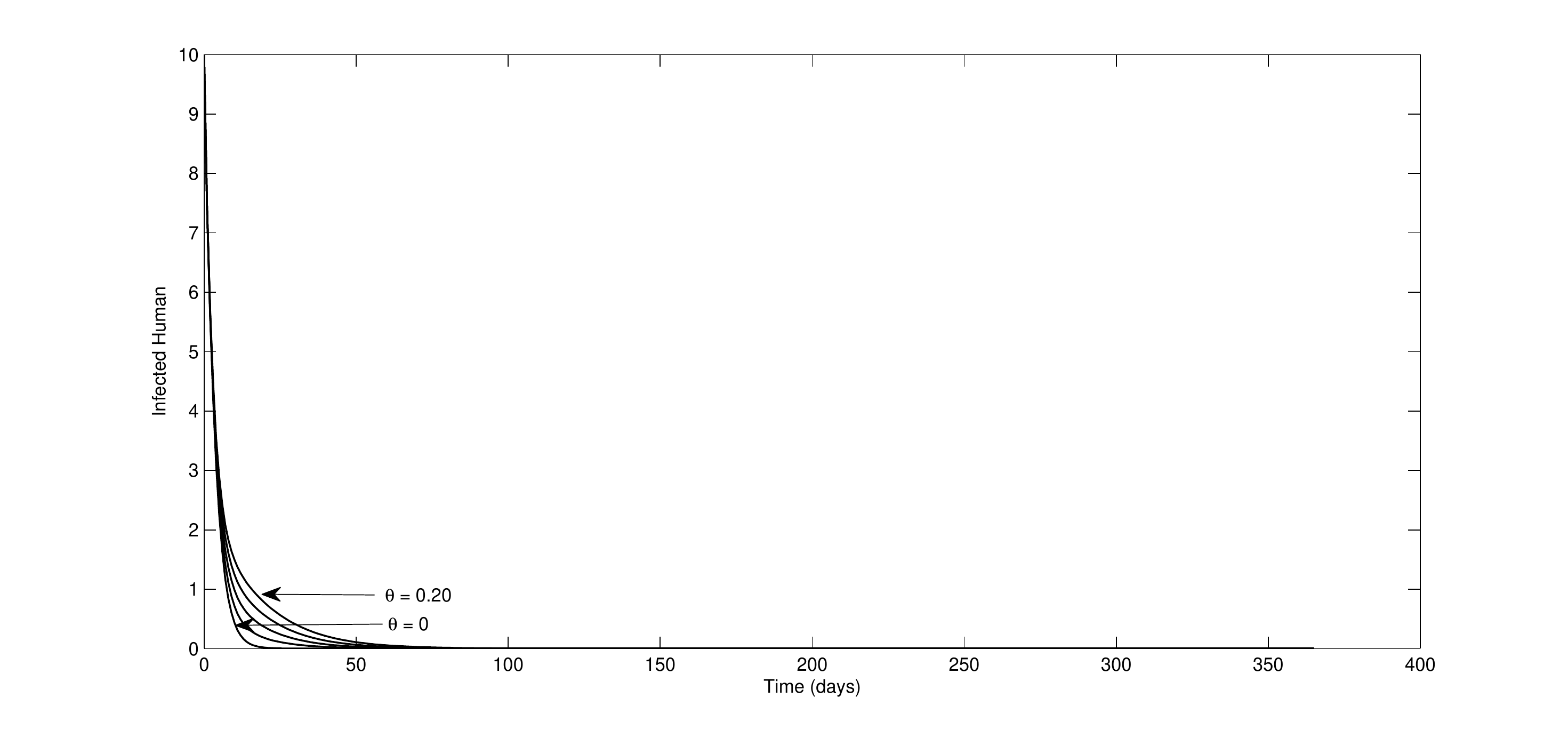}}
\caption{Infected humans in an outbreak considering 85\% of
susceptible population vaccinated ($\psi=0.85$)
and varying waning immunity $(\theta=0, 0.05, 0.10, 0.15, 0.20$).}
\label{cap7_theta_variation}
\end{figure}

Depending on the vaccine that will be available on the market,
it will be possible to choose or even combine features.
In the next section we define the vaccination process
as a control system.


\section{Vaccine as a control}
\label{sec:7:3}

In this section we consider a SIR model for humans and an ASI model for mosquitoes.
The parameters remain the same as in the previous section.
The vaccination is seen as a control variable to reduce
or even eradicate the disease. Let $u$
be the control variable: $0 \le u(t) \le 1$
denotes the percentage of susceptible individuals
that one decides to vaccinate at time $t$.
A random mass vaccination with waning immunity is selected. In this way,
a parameter $\theta$ associated to the control $u$ represents
the waning immunity process. Figure~\ref{cap7_modeloSIR_control_vaccine}
shows the epidemiological scheme for the human population.
Note that $R_h$ includes both vaccinated and naturally-immune individuals; 
only vaccinated immune individuals have waning immunity.
\begin{figure}[ptbh]
\center
\includegraphics[scale=0.4]{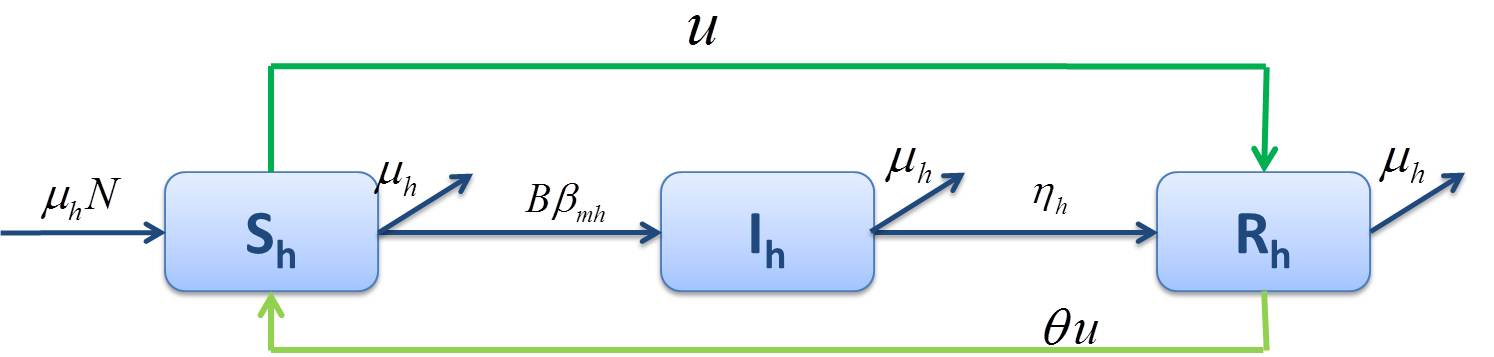}
\caption{\label{cap7_modeloSIR_control_vaccine} Epidemiological $SIR$ model
for the human population using the vaccine as a control.}
\end{figure}
The model is described by an initial value problem with a system
of six differential equations:
\begin{equation}
\label{cap7_ode}
\begin{cases}
\frac{dS_h}{dt} = \mu_h N_h- \left(B\beta_{mh}\frac{I_m}{N_h}+\mu_h+u\right)S_h+\theta u R_h\\
\frac{dI_h}{dt} = B\beta_{mh} \frac{I_m}{N_h} S_h -(\eta_h+\mu_h) I_h\\
\frac{dR_h}{dt} = \eta_h I_h + u S_h - \left(\theta u+\mu_h\right) R_h\\
\frac{dA_m}{dt} = \varphi \left(1-\frac{A_m}{k N_h}\right) (S_m+I_m) - \left(\eta_A+\mu_A\right) A_m\\
\frac{dS_m}{dt} = \eta_A A_m - \left(B \beta_{hm}\frac{I_h}{N_h}+\mu_m\right) S_m\\
\frac{dI_m}{dt} = B \beta_{hm}\frac{I_h}{N_h}S_m -\mu_m I_m.
\end{cases}
\end{equation}
The main aim is to study
the optimal vaccination strategy, considering both
the costs of treatment of infected individuals and the costs
of vaccination. The objective is to
\begin{equation}
\label{cap7_functional}
\text{ minimize } J[u]=\int_{0}^{t_f}\left[\gamma_D I_h(t)^2
+\gamma_V u(t)^2\right]dt,
\end{equation}
where $\gamma_D$ and $\gamma_V$ are positive constants
representing the weights of the costs
of treatment of infected people and vaccination, respectively.
We solve the problem using optimal control theory.


\subsection{Pontryagin's Maximum Principle}

Let us consider the following set of admissible control functions:
$$
\Delta=\{u(\cdot)\in L^{\infty}(0,t_f) \, | \, 0 \leq u(t) \leq 1, \forall t \in [0,t_f]\}.
$$

\begin{theorem}
The problem \eqref{cap7_ode}--\eqref{cap7_functional}
with initial conditions (see Table~\ref{chap7_initial_conditions}), admits
a unique optimal solution $\left(S_h^{*}(\cdot),I_h^{*}(\cdot),
R_h^{*}(\cdot),A_m^{*}(\cdot),S_m^{*}(\cdot),I_m^{*}(\cdot)\right)$
associated with an optimal control $u^{*}(\cdot)$ on $[0,t_f]$,
with a fixed final time $t_f$. Moreover, there exists adjoint functions
$\lambda_i^{*}(\cdot)$, $i=1,\ldots,6$, satisfying
\begin{equation}
\label{cap7_lambda_res}
\begin{tabular}{l}
$\left\{
\begin{array}{l}
\dot{\lambda}_{1}^{*}(t) = (\lambda_1-\lambda_2) \left(B \beta_{mh} \frac{I_m(t)}{N_h}\right)
+\lambda_1 \mu_h+(\lambda_1-\lambda_3)u(t)\\
\dot{\lambda}_{2}^{*}(t) = -2\gamma_D I_h(t)+\lambda_2(\eta_h+\mu_h)
-\lambda_3\eta_h+(\lambda_5-\lambda_6)\left(B\beta_{hm}\frac{S_m(t)}{N_h}\right)\\
\dot{\lambda}_{3}^{*}(t) = -\lambda_1\theta u(t)+\lambda_3(\mu_h+\theta u(t))\\
\dot{\lambda}_{4}^{*}(t) = \lambda_4 \varphi \frac{S_m(t)+I_m(t)}{k N_h}
+\lambda_4(\eta_A+\mu_A)-\lambda_5\eta_A\\
\dot{\lambda}_{5}^{*}(t) = -\lambda_4 \varphi \left(1-\frac{A_m(t)}{k N_h}\right)
+(\lambda_5-\lambda_6) B\beta_{hm} \frac{I_h(t)}{N_h}+\lambda_5\mu_m\\
\dot{\lambda}_{6}^{*}(t) = (\lambda_1-\lambda_2) \left(B\beta_{mh}\frac{S_h(t)}{N_h}\right)
-\lambda_4 \varphi \left(1-\frac{A_m(t)}{k N_h}\right) + \lambda_6 \mu_m\\
\end{array}
\right. $\\
\end{tabular}
\end{equation}
and the transversality conditions
$\lambda_{i}^{*}(t_f)=0$, i=1, \ldots 6. Furthermore,
\begin{equation}
\label{cap7_control}
u^{*}(t)=\min\left\{1,\max\left\{0,\frac{\left(\lambda_1(t)
-\lambda_3(t)\right)\left(S_h(t)-\theta R_h(t)\right)}{2\gamma_V}\right\}\right\}.
\end{equation}
\end{theorem}

\begin{proof}
The existence of optimal solutions
$\left(S_h^{*}(\cdot),I_h^{*}(\cdot),R_h^{*}(\cdot),A_m^{*}(\cdot),S_m^{*}(\cdot),I_m^{*}(\cdot)\right)$
associated to the optimal control $u^{*}(\cdot)$ comes from the convexity of the integrand
of the cost functional \eqref{cap7_functional} with respect to the control $u$ and the Lipschitz property
of the state system with respect to state variables $\left(S_h,I_h,R_h,A_m,S_m,I_m\right)$
(for more details, see \cite{Cesari1983,Silva2012}). According to the Pontryagin maximum principle
\cite{Pontryagin1962}, if $u^{*}(\cdot) \in \Delta$ is optimal for the problem considered,
then there exists a nontrivial absolutely continuous mapping
$\lambda:[0,t_f]\rightarrow \mathbb{R}$, $\lambda(t)=\left(\lambda_1(t),\lambda_2(t),\lambda_3(t),
\lambda_4(t),\lambda_5(t),\lambda_6(t)\right)$, called the adjoint vector, such that
\begin{equation}
\label{cap7_adjoint_system}
\dot{S}_h=\frac{\partial H}{\partial \lambda_1}, \quad
\dot{I}_h=\frac{\partial H}{\partial \lambda_2}, \quad
\dot{R}_h=\frac{\partial H}{\partial \lambda_3}, \quad
\dot{A}_m=\frac{\partial H}{\partial \lambda_4}, \quad
\dot{S}_m=\frac{\partial H}{\partial \lambda_5}, \quad
\dot{I}_m=\frac{\partial H}{\partial \lambda_6}
\end{equation}
and
\begin{equation}
\label{cap7_lambda_frac}
\dot{\lambda}_1=-\frac{\partial H}{\partial S_h}, \quad
\dot{\lambda}_2=-\frac{\partial H}{\partial I_h}, \quad
\dot{\lambda}_3=-\frac{\partial H}{\partial R_h}, \quad
\dot{\lambda}_4=-\frac{\partial H}{\partial A_m}, \quad
\dot{\lambda}_5=-\frac{\partial H}{\partial S_m}, \quad
\dot{\lambda}_6=-\frac{\partial H}{\partial I_m},
\end{equation}
where the Hamiltonian $H$ is defined by
\begin{equation*}
\begin{split}
H(S_h, I_h,R_h,&A_m,S_m,I_m,\lambda, u) =\gamma_D I_h^2+\gamma_V u^2
+\lambda_1\left(\mu_h N_h- \left(B\beta_{mh}\frac{I_m}{N_h}+\mu_h+u\right)S_h+\theta u R_h\right)\\
&+\lambda_2\left(B\beta_{mh} \frac{I_m}{N_h} S_h -(\eta_h+\mu_h) I_h\right)
+\lambda_3\left(\eta_h I_h + u S_h - (\theta u+\mu_h) R_h\right)\\
&+\lambda_4\left(\varphi \left(1-\frac{A_m}{k N_h}\right) (S_m+I_m) - \left(\eta_A+\mu_A\right) A_m\right)\\
&+\lambda_5\left(\eta_A A_m - \left(B \beta_{hm}\frac{I_h}{N_h}+\mu_m\right) S_m\right)
+\lambda_6\left(B \beta_{hm}\frac{I_h}{N_h}S_m -\mu_m I_m\right),
\end{split}
\end{equation*}
together with the minimality condition
\begin{multline}
\label{cap7_Hamil_min}
H\left(S_h^{*}(t),I_h^{*}(t),R_h^{*}(t),A_m^{*}(t),S_m^{*}(t),I_m^{*}(t),\lambda^{*}(t),u^{*}(t)\right)\\
=\min_{u} H\left(S_h^{*}(t),I_h^{*}(t),R_h^{*}(t),A_m^{*}(t),S_m^{*}(t),I_m^{*}(t),\lambda^{*}(t),u\right),
\end{multline}
satisfied almost everywhere on $[0,t_f]$.
Moreover, the transversality conditions
$\lambda_i(t_f)=0$ hold, $i=1,\ldots 6$.
System \eqref{cap7_lambda_res} is derived from \eqref{cap7_lambda_frac},
and the optimal control \eqref{cap7_control} comes from
the minimality condition \eqref{cap7_Hamil_min}.
\end{proof}


\subsection{Numerical simulations and discussion}
\label{numerical}

The simulations were carried out using the values of Section~\ref{sec:7:2}.
The values chosen for the weights in the objective functional 
\eqref{cap7_functional} were $\gamma_D=0.5$ and $\gamma_V=0.5$.
The system was normalized, putting all variables varying from 0 to 1.
It was considered that the waning immunity was at a rate of $\theta=0.05$.
The optimal control problem was solved using two methods: direct \cite{Betts2001,Trelat2005}
and indirect \cite{Lenhart2007}.
The direct method uses the cost functional \eqref{cap7_functional}
and the state system \eqref{cap7_ode} and was solved by
\texttt{DOTcvp} \cite{Dotcvp}. The indirect method
used is an iterative method with a Runge--Kutta scheme,
solved through \texttt{ode45} of \texttt{MatLab}.
Figure~\ref{cap7_optimal_control} shows the optimal control obtained by both methods.
Note that \texttt{DOTcvp} only gives the optimal control as a constant piecewise function.
Table~\ref{cap7_resultados} shows the costs obtained by the two
methods in both scenarios. The indirect method gives a lower cost.
This method uses more mathematical theory about the problem, such as the adjoint
system \eqref{cap7_adjoint_system} and optimal control expression \eqref{cap7_control}.
Therefore, it makes sense that the indirect method produces a better result.
\begin{figure}[ptbh]
\begin{center}
\subfloat[Epidemic scenario]{\label{cap7_optimal_control_epidemic}
\includegraphics[scale=0.52]{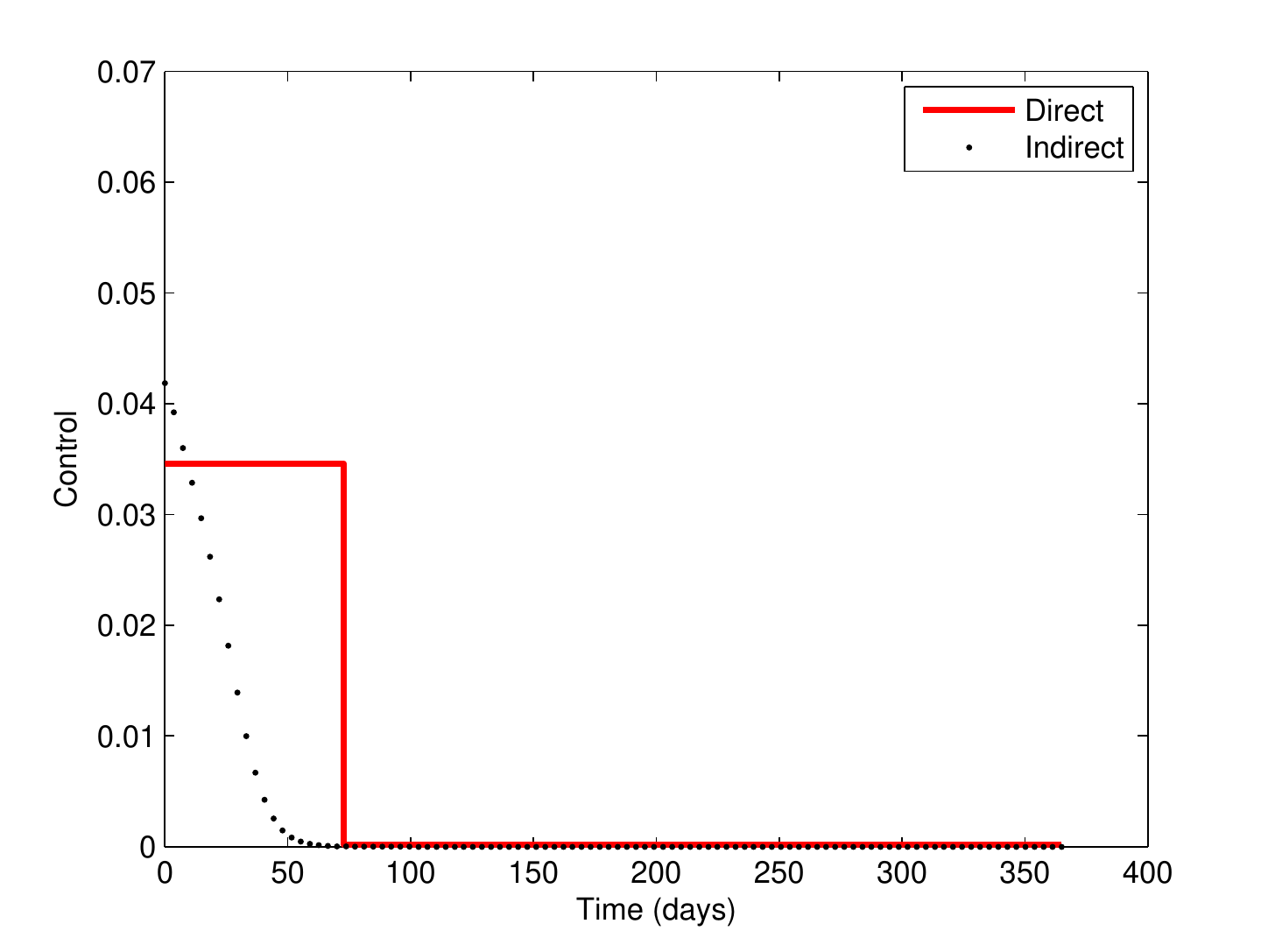}}
\subfloat[Endemic scenario]{\label{cap7_optimal_control_endemic}
\includegraphics[scale=0.52]{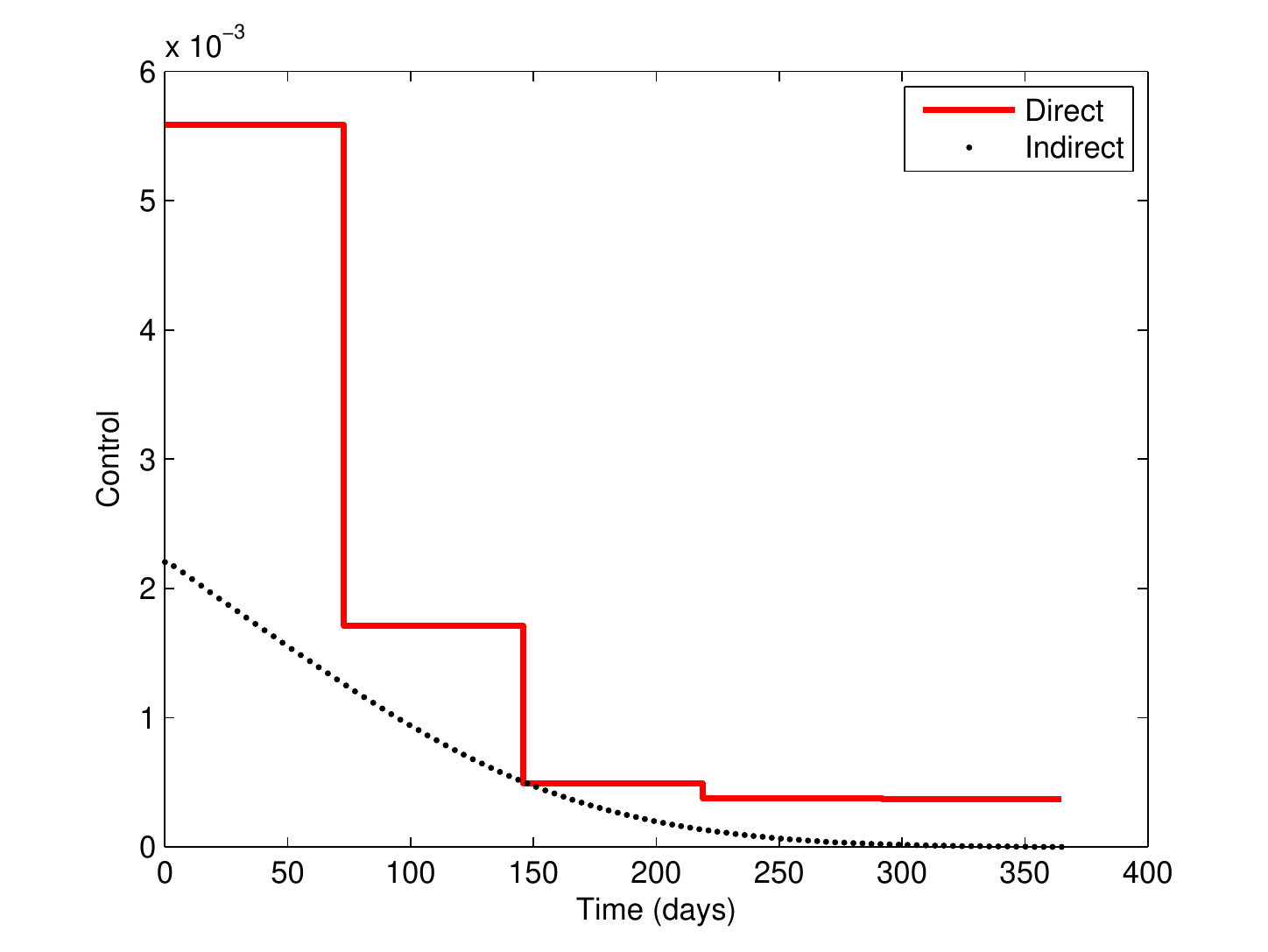}}
\caption{Optimal control obtained from direct and indirect approaches, in both epidemic and endemic scenarios.}
\label{cap7_optimal_control}
\end{center}
\end{figure}
\begin{table}
\begin{center}
\begin{tabular}{lccc}
\hline
Method & Epidemic scenario & Endemic scenario\\
\hline
Direct (DOTcvp) & 0.07505791 & 0.00189056\\
Indirect (backward-forward) & 0.06070556 & 0.00080618\\
\hline
\end{tabular}
\caption{Optimal values of the cost functional \eqref{cap7_functional}.} \label{cap7_resultados}
\end{center}
\end{table}

Using the optimal solution as reference, some tests were performed,
regarding infected individuals and costs, when no control ($u\equiv0$)
or upper control ($u\equiv1$) is applied.
Table~\ref{cap7_results_no_upper_control} shows the results for \texttt{DOTcvp}
in the three situations. In both scenarios, using the optimal strategy
of vaccination produces better costs with the disease, when compared to not doing anything.
Once there is no control, the number of infected humans is higher and produces a more expensive cost functional.
Figure~\ref{cap7_infected_no_upper_control} shows the number of infected
humans when different controls are considered. It is possible to see that
using the upper control, which means that everyone is vaccinated, implies
that just a few individuals were infected, allowing eradication of the
disease. Although the optimal control, in the sense of objective
\eqref{cap7_functional}, allows the occurrence of an outbreak,
the number of infected individuals is much lower when
compared with a situation where no one is vaccinated.
\begin{table}
\begin{center}
\begin{tabular}{lccc}
\hline
 & Epidemic scenario & Endemic scenario\\
\hline
optimal control & 0.07505791 & 0.00189056\\
no control & 0.32326592 & 0.01045990\\
upper control & 147.82500296 & 116.800000275\\ \hline
\end{tabular}
\caption{Values of the cost functional with optimal control,
no control ($u\equiv0$), and upper control ($u\equiv1$).} \label{cap7_results_no_upper_control}
\end{center}
\end{table}
\begin{figure}[ptbh]
\begin{center}
\subfloat[Epidemic scenario]{\label{cap7_epidemic}
\includegraphics[scale=0.52]{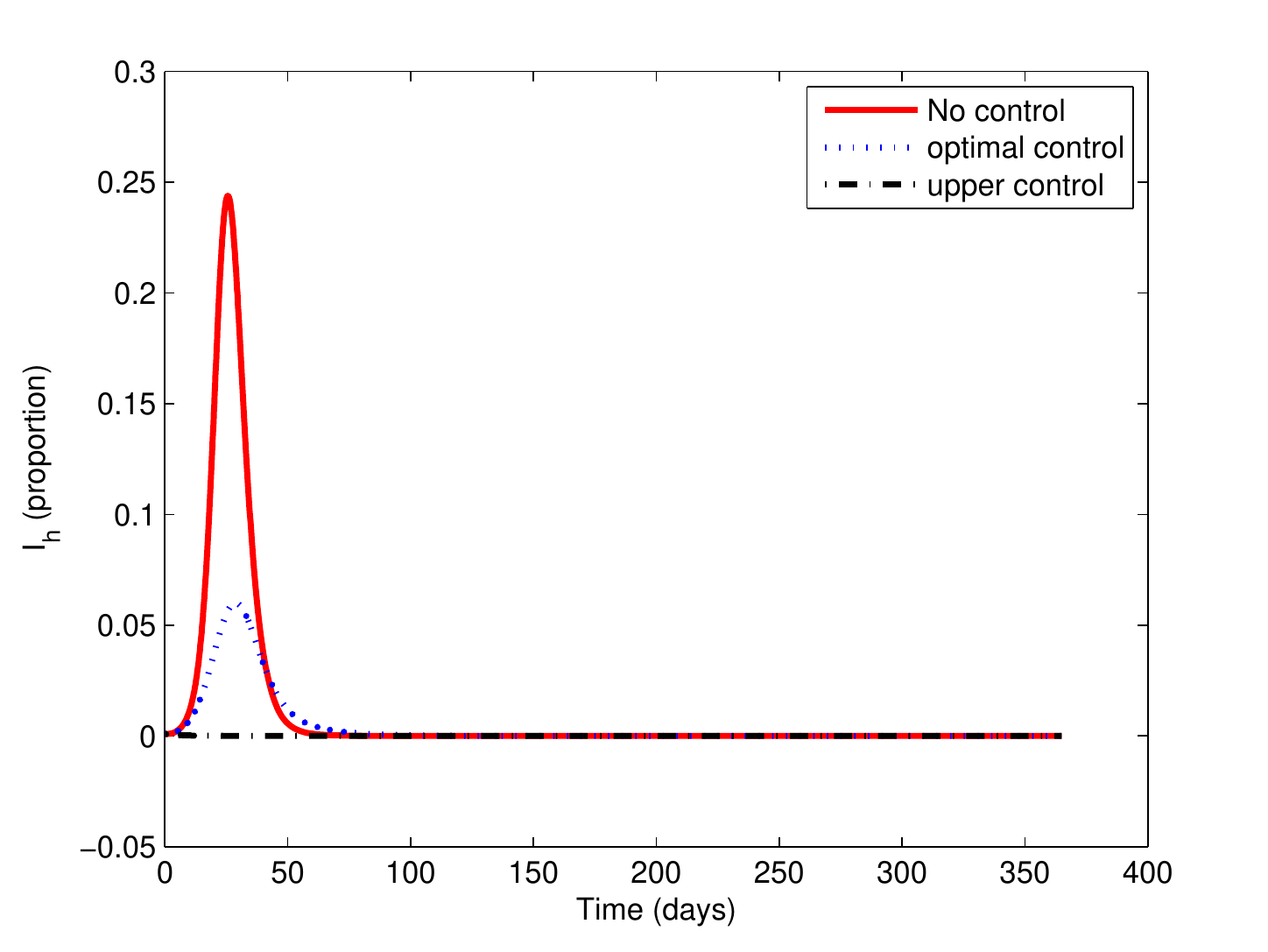}}
\subfloat[Endemic scenario]{\label{cap7_endemic}
\includegraphics[scale=0.52]{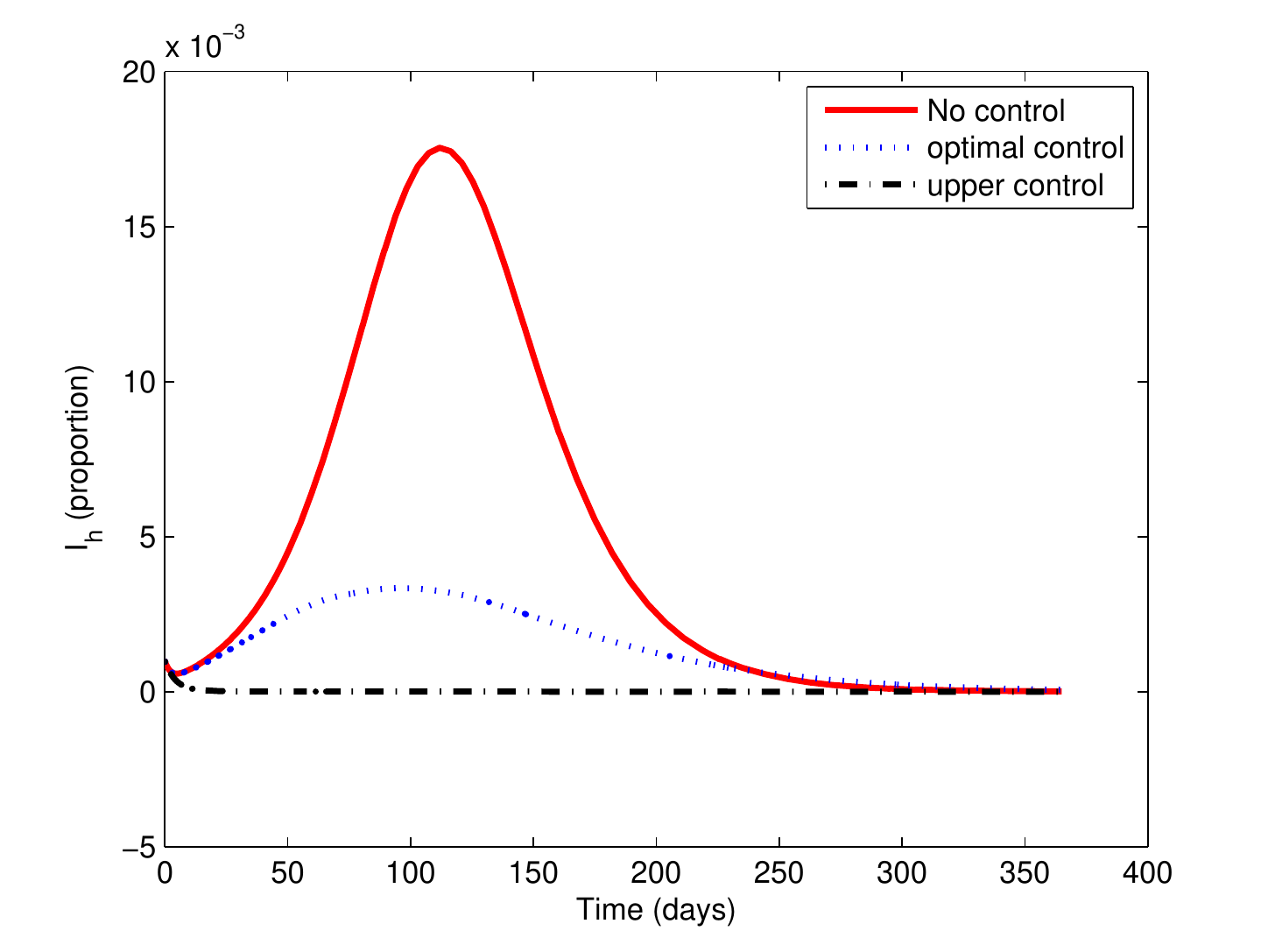}}
\caption{Infected humans when using the optimal vaccination strategy,
in terms of \eqref{cap7_functional}, no vaccination ($u \equiv 0$),
and vaccination of all susceptible population ($u \equiv 1$).}
\label{cap7_infected_no_upper_control}
\end{center}
\end{figure}

We conclude that, assuming a considerable efficacy level of the vaccine,
a vaccination campaign in the susceptible population
can quickly decrease the number of infected people.
Figures~\ref{eficacia_epidemico} and \ref{eficacia_endemic} 
show what happens to the optimal solution,
in the epidemic and endemic scenarios, respectively,
when the efficacy is changed. These figures
confirm an increasing of infected human and control application
with an increasing of the waning immunity rate $\theta$.
\begin{figure}[ptbh]
\begin{center}
\subfloat[Optimal control]{\label{new_controlo_eficacia_epidemico}
\includegraphics[width=0.51\textwidth]{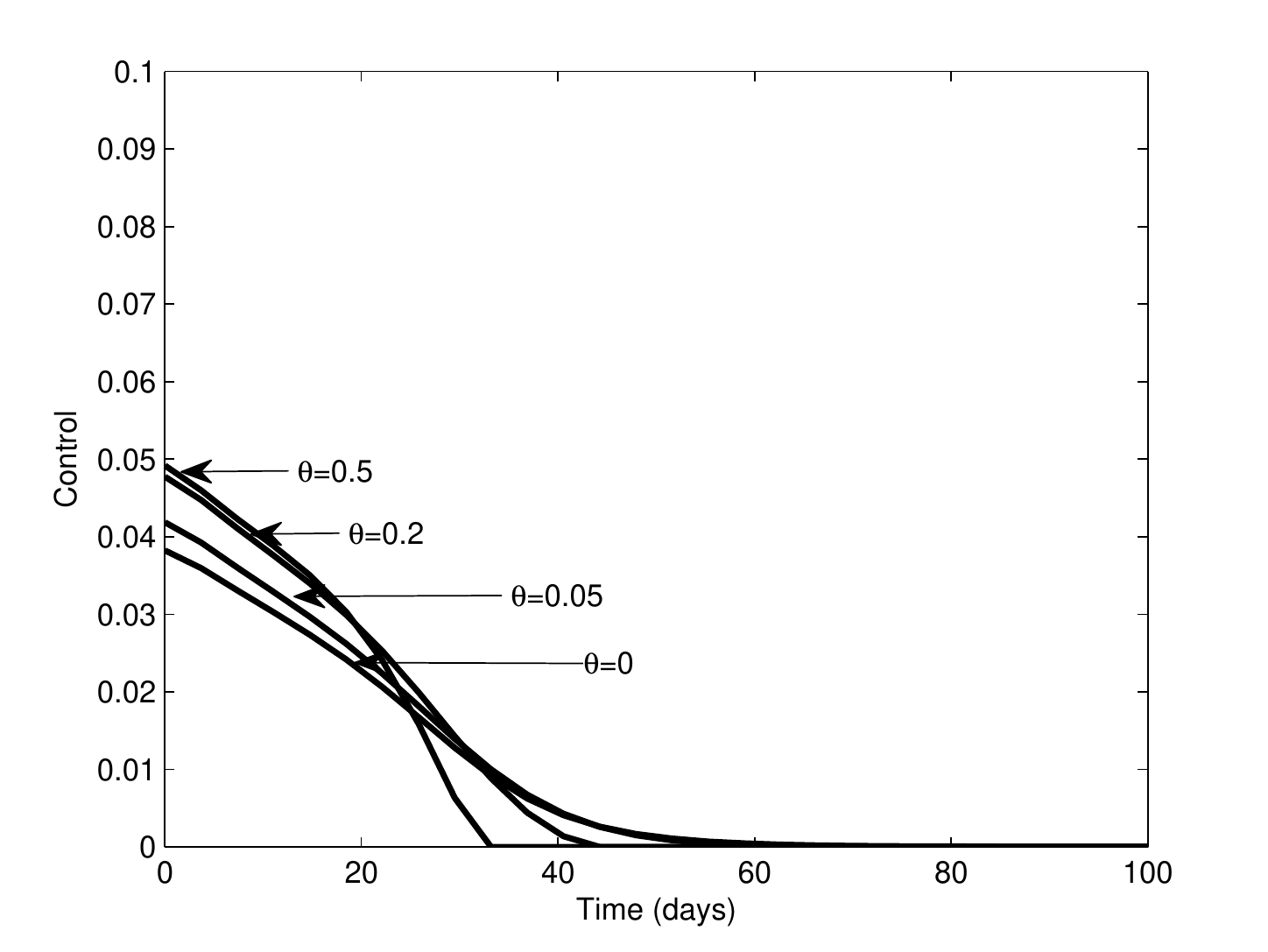}}
\subfloat[Number of infected]{\label{new_infetados_eficacia_epidemico}
\includegraphics[width=0.51\textwidth]{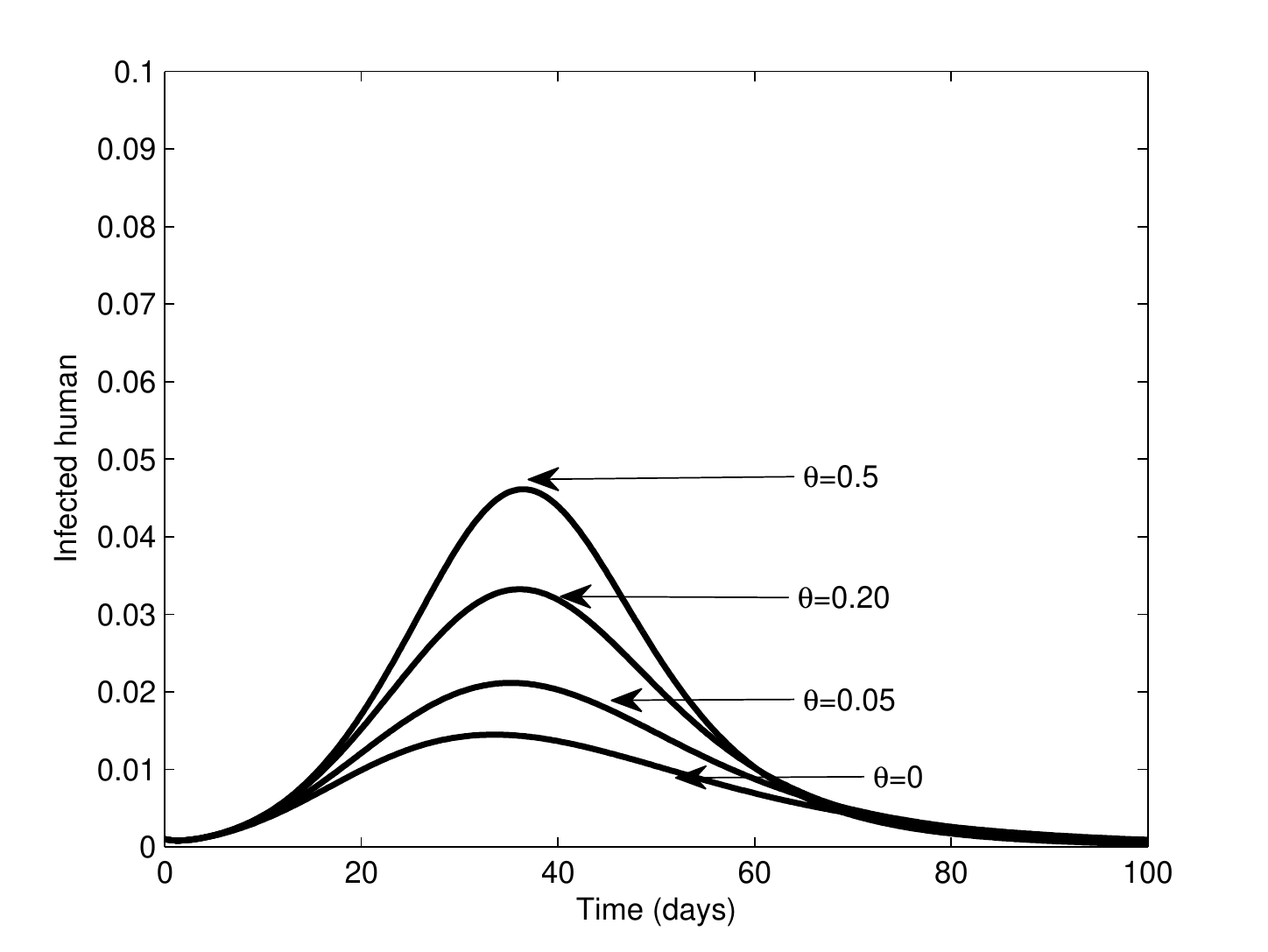}}
\caption{Optimal solutions for different vaccine efficacies (epidemic scenario).}
\label{eficacia_epidemico}
\end{center}
\end{figure}
\begin{figure}[ptbh]
\begin{center}
\subfloat[Optimal control]{\label{new_controlo_eficacia_endemic}
\includegraphics[width=0.51\textwidth]{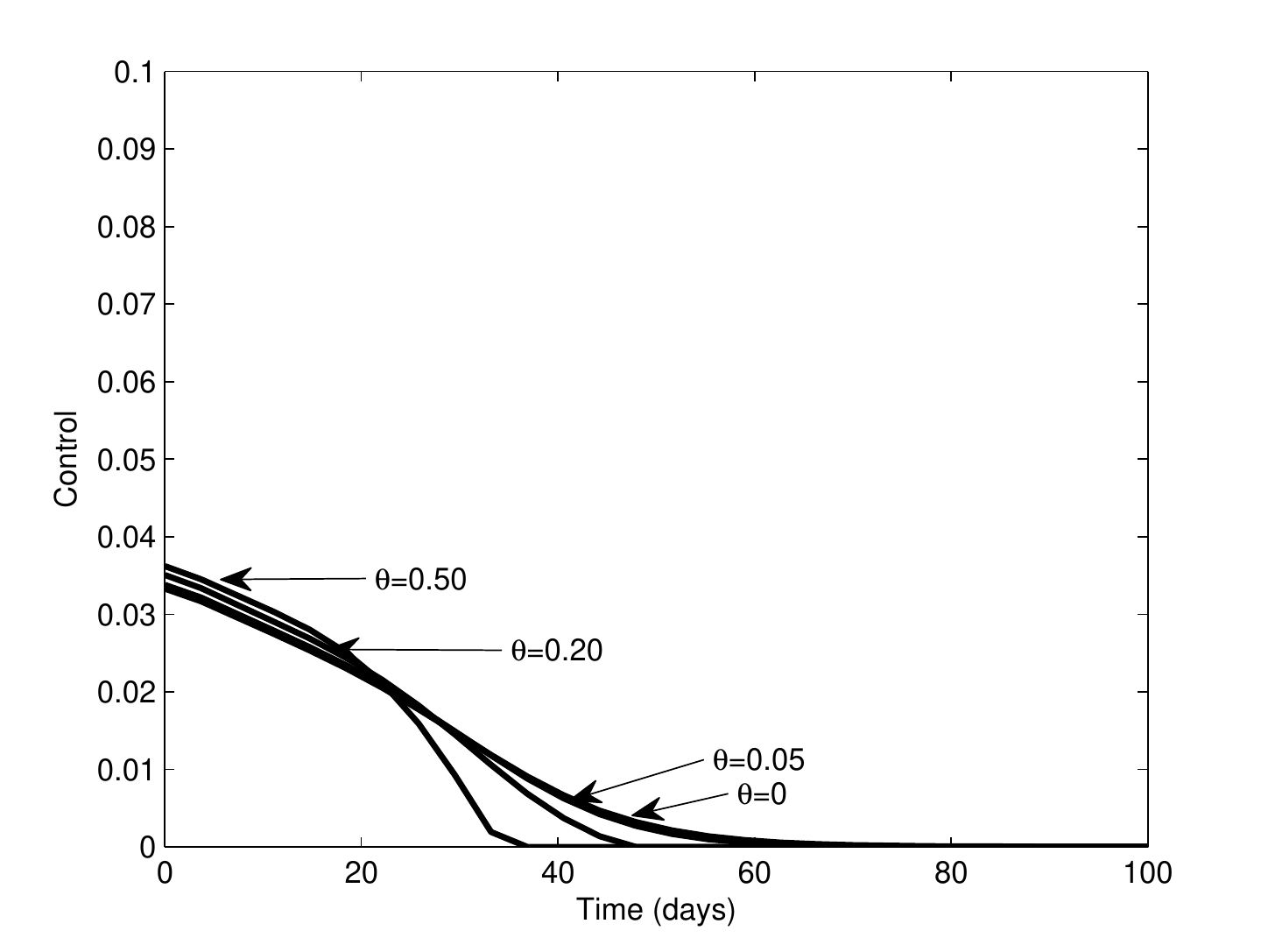}}
\subfloat[Number of infected]{\label{new_infetados_eficacia_endemic}
\includegraphics[width=0.51\textwidth]{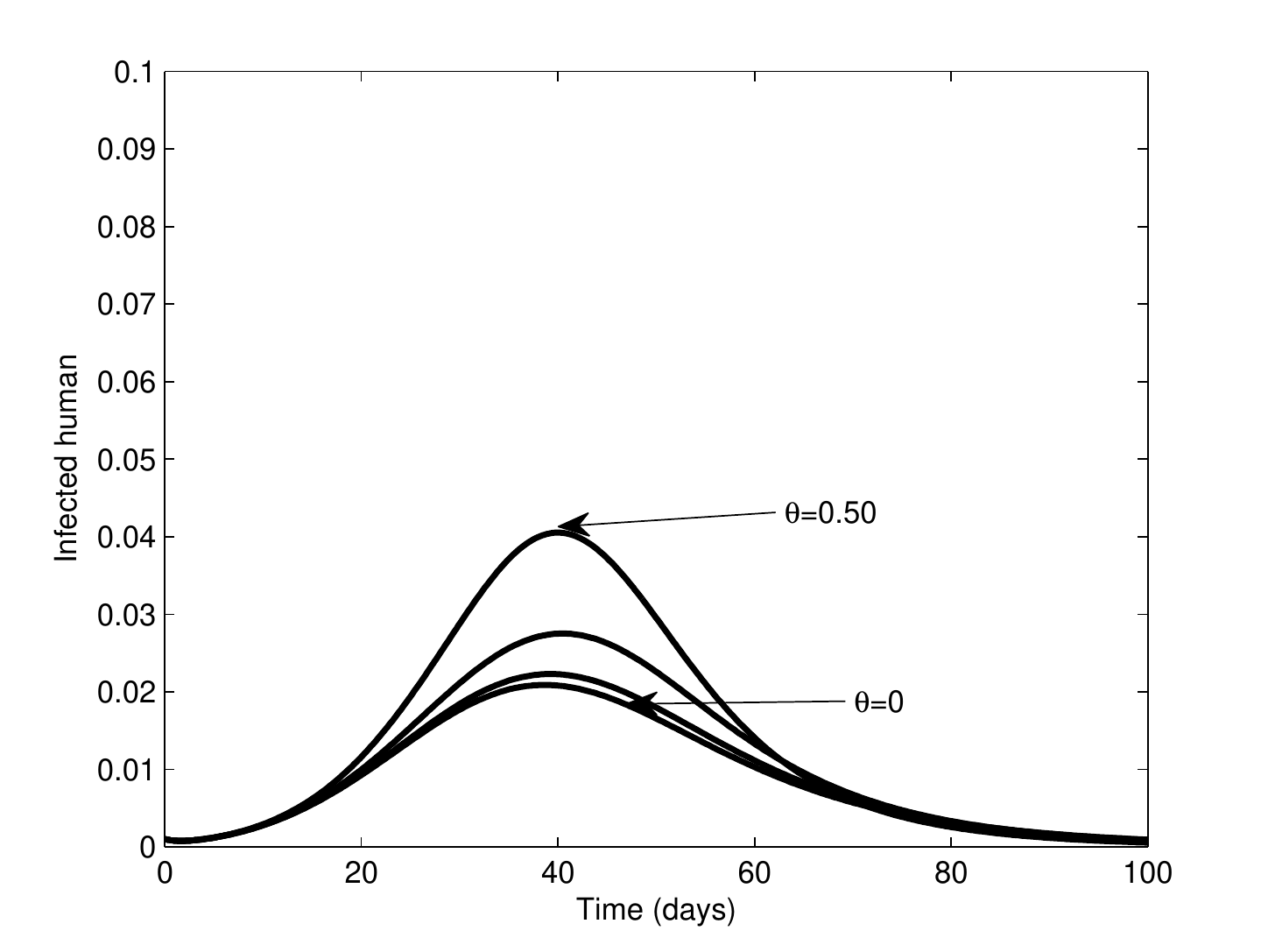}}
\caption{Optimal solutions for different vaccine efficacies (endemic scenario).}
\label{eficacia_endemic}
\end{center}
\end{figure}


\section{Conclusions}

The worldwide expansion of dengue fever is a growing health problem.
In Portugal, on October 2012, for the first time in its history,
an outbreak of dengue on the Madeira island occurred.
Besides \emph{Aedes albipictus}, the other vector that transmits the disease
is already in the Old Continent, namely in Italy and Spain.
Dengue vaccine is an urgent challenge that needs to be overcome.
This may be commercially available within a few years, when the researchers
will find a formula that protect against all four dengue viruses.
A vaccination program is seen as an important measure used
in infectious disease control and immunization and
eradication programs.

In the first part of the paper, different ways
of distributing the vaccine, as well as their features and some of coverage thresholds,
were introduced, in order to cover most of the future vaccine features.
The main goal of a vaccination program is to reduce the prevalence
of an infectious disease and ultimately to eradicate it.
It was shown that eradication success depends on the type of vaccine
as well as on the vaccination coverage. The imperfect vaccines may not
completely prevent infection but could reduce the probability of being infected,
thereby reducing the disease burden. In this study all the simulations were done
using epidemic and endemic scenarios to illustrate distinct realities.
A second analysis was made, using an optimal control approach. The vaccine behaved
as a new disease control variable and, when available, can be a promising
strategy to fight the disease.

Dengue is an infectious tropical disease difficult to prevent and manage.
Researchers agree that the development of a vaccine
for dengue is a question of high priority.
In the present study we have shown how a vaccine results
in reducing morbidity and, simultaneously, in a
reduction of the budget related with the disease.
As future work we intend to study the interaction of a dengue
vaccine with other kinds of controls already investigated in the literature,
such as insecticide and educational campaigns \cite{Sofia2012,Sofia2010c}.
It would be also interesting to investigate what happens if the final size 
of the epidemic is included in the objective functional.


\section*{Acknowledgements}

This work was supported by FEDER funds through
COMPETE --- Operational Programme Factors of Competitiveness
(``Programa Operacional Factores de Competitividade'')
and by Portuguese funds through the
Portuguese Foundation for Science and Technology
(``FCT --- Funda\c{c}\~{a}o para a Ci\^{e}ncia e a Tecnologia''),
within project PEst-C/MAT/UI4106/2011
with COMPETE number FCOMP-01-0124-FEDER-022690.
Rodrigues was also supported by FCT through the PhD grant SFRH/BD/33384/2008,
Monteiro by the R\&D unit Algoritmi and project FCOMP-01-0124-FEDER-022674,
and Torres by the Center for Research
and Development in Mathematics and Applications (CIDMA) and project
PTDC/MAT/113470/2009. The authors are very grateful
to two anonymous referees, for valuable remarks and comments,
which significantly contributed to the quality of the paper.



\end{document}